\documentclass[dratf]{fundam}

\usepackage{amssymb}
\usepackage{color,graphicx}
\usepackage{amsmath,amsfonts,amscd}
\usepackage{inputenc}
\usepackage{fancyhdr}

\usepackage[ruled,lined]{algorithm2e}
\usepackage[fleqn,tbtags]{mathtools}
\usepackage{empheq}

\newcommand{\Ini}{\textbf{Initialization: }}

\newcommand{\bfe}[1]{\textbf{\emph{#1}}}

\DeclareMathOperator{\Mon}{Mon}%
\DeclareMathOperator{\ind}{ind}%
\DeclareMathOperator{\lm}{lm}%
\DeclareMathOperator{\lc}{lc}%
\DeclareMathOperator{\lt}{lt}%
\DeclareMathOperator{\lcm}{lcm}%
\DeclareMathOperator{\Syz}{Syz}%

   \usepackage[colorlinks = true, linkcolor = black, urlcolor = black, citecolor = black, anchorcolor = black]{hyperref}

\begin{document}

\setcounter{page}{83}
\publyear{2021}
\papernumber{2093}
\volume{184}
\issue{2}

      \finalVersionForIOS

\title{Right Buchberger Algorithm over Bijective Skew $PBW$ Extensions}

\author{William Fajardo\thanks{Address  for correspondence: Seminario de \'Algebra Constructiva - $\text{SAC}^2$,
                  Departamento de Matem\'aticas, Universidad Nacional de Colombia, Bogot\'a, Colombia. \newline \newline
          \vspace*{-6mm}{\scriptsize{Received February 2021; \ revised November 2021.}}}
\\
Seminario de \'Algebra Constructiva - $\text{SAC}^2$\\
Departamento de Matem\'aticas\\
Universidad Nacional de Colombia, Bogot\'a, Colombia\\
wafajardoc@unal.edu.co
}

\runninghead{W. Fajardo}{Right Buchberger Algorithm over Bijective Skew $PBW$ Extensions}

\maketitle

\vspace*{-4mm}
\begin{abstract}
\noindent In this paper we present a right version of the Buchberger algorithm over skew Poincar\'e-Birkhoff-Witt extensions (skew PBW extensions for short) defined by Gallego and Lezama \cite{LezamaGallego}. This algorithm is an adaptation of the left case given in \cite{Fajardo3}. In particular, we developed a right version of the division algorithm and from this we built the right Gröbner bases theory over bijective skew $PBW$ extensions. The algorithms were implemented in the SPBWE library developed in \textrm{Maple}, this paper includes an application of these to the membership problem. The theory developed here is fundamental to complete the \texttt{SPBWE} library and thus be able to implement various homological applications that arise as result of obtaining the right Gröbner bases over skew $PBW$ extensions.
\end{abstract}

\begin{keywords}
 Non-commutative computational algebra, skew $PBW$ extensions, Buchberger algorithm, Gröbner bases, \texttt{SPBWE} library, {\rm Maple}.
\end{keywords}

\section{Skew $PBW$ extensions}

In this section we introduce the \textit{bijective skew $PBW$ extensions} whose are the fundamental topic in this paper. Skew $PBW$ extensions include well known classes of Ore algebras, operator algebras and also a lot of quantum rings and algebras. The skew $PBW$ extensions have been extensively studied, see \cite{Fajardo3}, these are being implemented in the \texttt{SPBWE} library developed in \textrm{Maple}, see \cite{Fajardo1} and  \cite{Fajardo}. The main purpose of this paper is to present the theory needed to develop the right Gröbner theory and generate their respective algorithms implemented in \textrm{Maple} through \texttt{SPBWE} library., i.e., implementing the division algorithm and Buchberger algorithm in the right case, similar works have been implemented, see Fajardo \cite{Fajardo1} and \cite{Fajardo}, Fajardo-Lezama \cite{Fajardo2}, Gasiorek et al., \cite{Gasiorek}, Simson-Wojew\'odzki \cite{Simson-Wojewodzky}, Simson \cite{Simson1}, \cite{Simson2} and  \cite{Simson3}.

\begin{definition}\label{gpbwextension}
Let $R$ and $A$ be rings, we say that $A$ is a skew $PBW$ extension of $R$ {\rm(}also called
$\sigma$-$PBW$ extension{\rm)}, if the following conditions hold:
\begin{enumerate}
\itemsep=0.9pt
\item[\rm (i)]$R\subseteq A$.
\item[\rm (ii)]There exist finitely many elements $x_1,\dots ,x_n\in A$ such that $A$ is a left $R$-free module with basis

\smallskip\centerline{$\Mon(A):=\Mon\{x_1,\dots,x_n\}=\{x^{\alpha}:=x_1^{\alpha_1}\cdots
x_n^{\alpha_n}|\alpha=(\alpha_1,\dots ,\alpha_n)\in \mathbb{N}^n\}.$}

\item[\rm (iii)]For every $1\leq i\leq n$ and $r\in R-\{0\}$ there exists $c_{i,r}\in R-\{0\}$ such that\vspace*{-1mm}
\begin{equation}\label{sigmadefinicion1}
x_ir-c_{i,r}x_i\in R. \vspace*{-1mm}
\end{equation}
\item[\rm (iv)]For every $1\leq i,j\leq n$ there exists $c_{i,j}\in R-\{0\}$ such that\vspace*{-1mm}
\begin{equation}\label{sigmadefinicion2}
x_jx_i-c_{i,j}x_ix_j\in R+Rx_1+\cdots +Rx_n.\vspace*{-1mm}
\end{equation}
\end{enumerate}
In this case the extension is denoted by $A=\sigma(R)\langle x_1,\dots ,x_n\rangle$, and $R$ is called the ring of coefficients of the extension $A$.
\end{definition}
\begin{remark}\label{notesondefsigampbw}
Each element $f\in A-\{0\}$ has a unique representation in the form $f=c_1X_1+\cdots+c_tX_t$, with $c_i\in R-\{0\}$ and $X_i\in \Mon(A)$, $1\leq i\leq t$.
\end{remark}
The following proposition (see \cite{Fajardo3}, Proposition 1.1.3) justifies the notation given in Definition \ref{gpbwextension} of the skew $PBW$ extensions.
\begin{proposition}\label{sigmadefinition}
Let $A=\sigma(R)\langle x_1,\ldots,x_n\rangle$ be a skew $PBW$ extension of $R$. Then, for $1\leq i\leq n$, there exist an injective ring
endomorphism $\sigma_i:R\rightarrow R$ and a $\sigma_i$-derivation $\delta_i:R\rightarrow R$ such
that

\smallskip\centerline{$x_ir=\sigma_i(r)x_i+\delta_i(r)$,}

\smallskip\noindent  for every $r\in R$.
\end{proposition}
\begin{definition}\label{sigmapbwbijectve}
Let $A=\sigma(R)\langle x_1,\ldots,x_n\rangle$ be a skew $PBW$ extension of $R$. $A$ is called bijective if $\sigma_i$ is bijective for every $1\leq i\leq n$ and $c_{i,j}$ is invertible for any $1\leq i,j\leq n$.
\end{definition}
\begin{definition}\label{1.1.6}
Let $A=\sigma(R)\langle x_1,\ldots,x_n\rangle$ be a skew $PBW$ extension of $R,$ with endomorphisms $\sigma_i, 1\leq i\leq n$. We will use the following notation.
\begin{enumerate}
\item[\rm (i)]For $\alpha=(\alpha_1,\dots,\alpha_n)\in \mathbb{N}^n$, $\sigma^{\alpha}:=\sigma_1^{\alpha_1}\cdots \sigma_n^{\alpha_n},$  $|\alpha|:=\alpha_1+\cdots+\alpha_n$ and if $A$ is bijective $\sigma^{-\alpha}\!:=\!\sigma_n^{-\alpha_n}\cdots \sigma_1^{-\alpha_1}.],$ Moreover, if $\beta\!=\!(\beta_1,...,\beta_n)\in \mathbb{N}^n$, then $\alpha+\beta\!:=\!(\alpha_1+\beta_1,...,\alpha_n+\beta_n)$.
\item[\rm (ii)]For $X=x^{\alpha}\in \Mon(A)$, $\exp(X):=\alpha$ and $\deg(X):=|\alpha|$.
\item[\rm (iii)]Let $0\neq f\in A$; if $f=c_1X_1+\cdots +c_tX_t$, with $X_i\in \Mon(A)$ and $c_i\in R-\{0\}$, then $\deg(f):=\max\{\deg(X_i)\}_{i=1}^t.$
\end{enumerate}
\end{definition}
The following characterization of skew $PBW$ extensions was given in \cite{gallego-thesis}.
\begin{theorem}\label{coefficientes}
Let $A$ be a ring of a left polynomial type over $R$ with respect to $\{x_1,\dots,x_n\}$. $A$ is a skew
$PBW$ extension of $R$ if and only if the following conditions hold:
\begin{enumerate}
\item[\rm (a)]For every $x^{\alpha}\in \Mon(A)$ and every $0\neq
r\in R$ there exist unique elements $r_{\alpha}:=\sigma^{\alpha}(r)\in R-\{0\}$ and $p_{\alpha
,r}\in A$ such that
\begin{equation}\label{611}
x^{\alpha}r=r_{\alpha}x^{\alpha}+p_{\alpha , r},
\end{equation}
where $p_{\alpha ,r}=0$ or $\deg(p_{\alpha ,r})<|\alpha|$ if $p_{\alpha , r}\neq 0$. Moreover, if
$r$ is left invertible, then $r_\alpha$ is left invertible.

\item[\rm (b)]For every $x^{\alpha},x^{\beta}\in \Mon(A)$ there
exist unique elements $c_{\alpha,\beta}\in R$ and $p_{\alpha,\beta}\in A$ such that\vspace*{-1mm}
\begin{equation}\label{612}
x^{\alpha}x^{\beta}=c_{\alpha,\beta}x^{\alpha+\beta}+p_{\alpha,\beta},\vspace*{-1mm}
\end{equation}
where $c_{\alpha,\beta}$ is left invertible, $p_{\alpha,\beta}=0$ or
$\deg(p_{\alpha,\beta})<|\alpha+\beta|$ if $p_{\alpha,\beta}\neq 0$.
\end{enumerate}
\end{theorem}

\begin{remark}\label{identities}\hfill
\begin{enumerate}
\item[\rm (i)] Let $\theta,\gamma,\beta\in \mathbb{N}^n$ and $c\in R.$ Then we have the following identities:
    \begin{align}
    \sigma^\theta(c_{\gamma,\beta})c_{\theta,\gamma+\beta}&    =c_{\theta,\gamma}c_{\theta+\gamma,\beta},\\
    \sigma^\theta(\sigma^\gamma (c))c_{\theta,\gamma} &
    =c_{\theta,\gamma}\sigma^{\theta+\gamma}(c). \vspace*{-1mm}
    \end{align}
\item[\rm (ii)] One concludes from Theorem \ref{coefficientes} that if $A$ is bijective, then $c_{\alpha,\beta}$ is invertible for any $\alpha,\beta\in \mathbb{N}^n$.
\end{enumerate}
\end{remark}

\section{Orders on $\Mon(A^{m})$}

In this section we will compile some results taken from \cite{Fajardo3} that will be used in the theory of reduction and the theory of Gröbner for the right case.
\begin{definition}
\begin{enumerate}
\item[\rm (a)] We define in $\Mon(A)$ the \textit{deglex order} by the formulas:
\begin{center}
$x^{\alpha}\succeq x^{\beta}\Longleftrightarrow
\begin{cases}
x^{\alpha}=x^{\beta}\\
\text{or} & \\
x^{\alpha}\neq x^{\beta}\, \text{but} \, |\alpha|> |\beta| & \\
\text{or} & \\
x^{\alpha}\neq x^{\beta},|\alpha|=|\beta|\, \text{but $\exists$ $i$ with} &
\alpha_1=\beta_1,\dots,\alpha_{i-1}=\beta_{i-1},\alpha_i>\beta_i.
\end{cases}$
\end{center}
It is clear that the deglex order is a total order.
\eject
\item[\rm (b)] If $x^{\alpha}\succeq x^{\beta}$ and $x^{\alpha}\neq x^{\beta}$ we write $x^{\alpha}\succ x^{\beta}$.
\item[\rm (c)] Assume that the element $f\in A\setminus\{0\}$ has the unique form $f=c_1x^{\alpha_1}+\cdots+c_tx^{\alpha_t}$, with $c_i\in R-\{0\}$, $1\leq i\leq t$, and $x^{\alpha_1}\succ\cdots\succ x^{\alpha_t}$. We define the monomial $x^{\alpha_1}$ to be the \textit{leader monomial} of $f$ and we write $\lm(f):=x^{\alpha_1}$ ; $c_1$ is the \textit{leader coefficient} of $f$, $\lc(f):=c_1$ and $c_1x^{\alpha_1}$ is the \textit{leader term} of $f$ denoted by $\lt(f):=c_1x^{\alpha_1}$. If $f=0$, we define $\lm(0):=0,\lc(0):=0,\lt(0):=0$ and we set $X\succ 0$ for any $X\in \Mon(A)$.
\end{enumerate}
\end{definition}

\subsection{Monomial orders in skew $PBW$ extensions}\label{sec3}

Let $A=\sigma(R)\langle x_1, \dots , x_n\rangle$ be a skew $PBW$ extension of $R$ and let $\succeq$ be a total order defined on $\Mon(A)$. If $x^{\alpha}\succeq x^{\beta}$ but $x^{\alpha}\neq
x^{\beta}$ we will write $x^{\alpha}\succ x^{\beta}$. $x^{\beta}\preceq x^{\alpha}$ means that $x^{\alpha}\succeq x^{\beta}$. Let $f\neq 0$ be a polynomial of $A.$ If
\begin{center}
$f=c_1X_1+\cdots +c_tX_t$,
\end{center}
with $c_i\in R-\{0\}$ and $X_1\succ \cdots \succ X_t$ are the monomials of $f$, then $\lm(f):=X_1$ is the \textit{leading monomial} of $f$, $\lc(f):=c_1$ is the \textit{leading coefficient}
of $f$ and $\lt(f):=c_1X_1$ is the \textit{leading term} of $f$. If $f=0$, we define $\lm(0):=0,\lc(0):=0,\lt(0):=0$, and we set $X\succ 0$ for any $X\in \Mon(A)$. Thus, we extend $\succeq$ to $\Mon(A)\cup \{0\}$.
\begin{definition}\label{monomialorder}
Let $\succeq$ be a total order on $\Mon(A)$, it is said that $\succeq$
is a monomial order on $\Mon(A)$ if the following conditions hold:
\begin{enumerate}
\item[\rm (i)]For every $x^{\beta},x^{\alpha},x^{\gamma},x^{\lambda}\in \Mon(A)$
\begin{center}
$x^{\beta}\succeq x^{\alpha}$ $\Rightarrow$
$\lm(x^{\gamma}x^{\beta}x^{\lambda})\succeq
\lm(x^{\gamma}x^{\alpha}x^{\lambda})$.
\end{center}

\item[\rm (ii)]$x^{\alpha}\succeq 1$, for every $x^{\alpha}\in
\Mon(A)$.
\item[\rm (iii)]$\succeq$ is degree compatible, i.e., $|\beta|\geq |\alpha|\Rightarrow x^{\beta}\succeq
x^{\alpha}$.
\end{enumerate}
\end{definition}
Monomial orders are also called \textit{admissible orders}. It is worth noting that every monomial order on $\Mon(A)$ is a well order. Thus, there are not infinite decreasing chains in $\Mon(A)$. From now on we will assume that $\Mon(A)$ is endowed with some monomial order.
\begin{definition}
Let $x^{\alpha},x^{\beta}\in \Mon(A)$, we say that $x^{\alpha}$
divides $x^{\beta}$, denoted by $x^{\alpha}|x^{\beta}$, if there
exists $x^{\gamma},x^{\lambda}\in \Mon(A)$ such that
$x^{\beta}=\lm(x^{\gamma}x^{\alpha}x^{\lambda})$. We will say also
that any monomial $x^{\alpha}\in \Mon(A)$ divides the polynomial
zero.
\end{definition}
The condition (iii) of Definition \ref{monomialorder} is needed in the proof of the following proposition (see \cite{Fajardo3}, Proposition 13.1.4), and this one will be used in right Division Algorithm (Theorem \ref{algdivformodules}).
\begin{proposition}\label{11.1.4}
Let $A$ be a bijective skew PBW extension and $x^{\alpha},x^{\beta}\in \Mon(A)$ and $f,g\in A-\{0\}$. Then,
\begin{enumerate}
\item[\rm (a)]
$\lm(x^{\alpha}g)=\lm(x^{\alpha}\lm(g))=x^{\alpha+\exp(\lm(g))}$,
i.e., $\exp(\lm(x^{\alpha}g))=\alpha+\exp(\lm(g)$. In particular,
\begin{center}
$\lm(\lm(f)\lm(g))=x^{\exp(\lm(f))+\exp(\lm(g))}$, i.e.,

$\exp(\lm(\lm(f)\lm(g)))=\exp(\lm(f))+\exp(\lm(g))$
\end{center}
and
\begin{equation}
\lm(x^{\alpha}x^{\beta})=x^{\alpha+\beta}, \ i.e.,\
\exp(\lm(x^{\alpha}x^{\beta}))=\alpha+\beta .\label{divrelation}
\end{equation}
\item[\rm (b)]The
following conditions are equivalent:
\begin{enumerate}
\item[\rm (i)]$x^{\alpha}|x^{\beta}$.
\item[\rm (ii)]There exists a unique $x^{\theta}\in \Mon(A)$ such that
$x^{\beta}=\lm(x^{\theta}x^{\alpha})=x^{\theta+\alpha}$ and hence
$\beta=\theta+\alpha$.
\item[\rm (iii)]There exists a unique $x^{\theta}\in \Mon(A)$ such that
$x^{\beta}=\lm(x^{\alpha}x^{\theta})=x^{\alpha+\theta}$ and hence
$\beta=\alpha+\theta$.
\item[\rm (iv)]$\beta_i\geq \alpha_i$ for $1\leq i\leq n$, with
$\beta:=(\beta_1,\dots,\beta_n)$ and
$\alpha:=(\alpha_1,\dots,\alpha_n)$.
\end{enumerate}
\end{enumerate}
\end{proposition}
\begin{proof}
Apply [3; Proposition 13.1.4].
\end{proof}
\begin{remark}\hfill
\begin{itemize}
  \item[\rm (i)] Let $\succeq$ be the monomial order on $\Mon(A).$ If there exists $f=x^{\gamma_1}c_1+\cdots+x^{\gamma_t}c_t\in A-\{0\}$ such that $x^{\beta}=x^{\alpha}f$ or $x^{\beta}=fx^{\alpha}$, then by Proposition \ref{11.1.4}, $x^{\beta}=x^{\alpha+\gamma_1}$, i.e., $x^{\alpha}|x^{\beta}.$
  \item[\rm (ii)] We note that there exists a least common multiple of two elements of $\Mon(A)$: in fact, let $x^{\alpha},x^{\beta}\in \Mon(A)$, then $\lcm(x^{\alpha},x^{\beta})=x^{\gamma}\in \Mon(A)$, where $\gamma=(\gamma_1,\dots,\gamma_n)$ with $\gamma_i:=\max\{\alpha_i,\beta_i\}$ for each $1\leq i\leq n$.
\end{itemize}
\end{remark}

\subsection{Monomial orders on $\Mon(A^m)$}
We will often represent the elements of $A^m$ also as row vectors, in case when if this does not causa confusion. We recall that the canonical basis of the free $A$-module $A^m$ is
\begin{center}
$\textbf{\emph{e}}_1=(1,0,\dots
,0),\textbf{\emph{e}}_2=(0,1,0,\dots, 0),\dots
,\textbf{\emph{e}}_m=(0,0,\dots ,1)$.
\end{center}
\begin{definition}
A monomial in $A^m$ is a vector $\bfe{X}=X\bfe{e}_i$, where
$X=x^{\alpha}\in \Mon(A)$ and $1\leq i\leq m$, i.e.,\vspace*{-1mm}
\begin{center}
$\bfe{X}=X\bfe{e}_i=(0,\dots ,X,\dots ,0)$,
\end{center}
where $X$ is in the $i$-th position, named the index of
$\bfe{X}$, $\ind(\bfe{X}):=i$. A term is a vector
$c\bfe{X}$, where $c\in R$. The set of monomials of $A^m$ will
be denoted by $\Mon(A^m)$. Let
$\bfe{Y}=Y\bfe{e}_j\in \Mon(A^m)$, we say that $\bfe{X}$
divides $\bfe{Y}$ if $i=j$ and $X$ divides $Y$. We will say
that any monomial $\bfe{X}\in \Mon(A^m)$ divides the null vector
$\bfe{\emph{0}}$. The least common multiple of $\bfe{X}$ and
$\bfe{Y}$, denoted by $\lcm(\bfe{X},\bfe{Y})$, is
$\bfe{\emph{0}}$ if $i\neq j$, and $U\bfe{e}_i$, where
$U=\lcm(X,Y)$, if $i=j$. Finally, we define
\begin{center}
$\exp(\bfe{X}):=\exp(X)=\alpha$ and
$\deg(\bfe{X}):=\deg(X)=|\alpha|$.
\end{center}
\end{definition}
\eject

\noindent We now define monomial orders on $\Mon(A^m)$.
\begin{definition}\label{DefOrderMon}
A monomial order on $\Mon(A^m)$ is a total order $\succeq$
satisfying the following three conditions:
\begin{enumerate}
\item[\rm(i)] $\lm(x^{\beta}x^{\alpha})\bfe{e}_{i}\succeq x^{\alpha}\bfe{e}_{i}$, for every
monomial $\bfe{X}=x^{\alpha}\bfe{e}_{i}\in \Mon(A^{m})$ and
any monomial $x^{\beta}$ in $\Mon(A)$.
\item[\rm(ii)] If $\bfe{Y}=x^{\beta}\bfe{e}_{j}\succeq \bfe{X}=x^{\alpha}\bfe{e}_{i}$, then
$\lm(x^{\gamma}x^{\beta})\bfe{e}_{j}\succeq
\lm(x^{\gamma}x^{\alpha})\bfe{e}_{i}$ for every monomial $x^{\gamma}\in
\Mon(A)$.
\item[\rm (iii)]$\succeq$ is degree compatible, i.e., $\deg(\bfe{X})\geq \deg(\bfe{Y})\Rightarrow \bfe{X}\succeq
\bfe{Y}$.
\end{enumerate}
If $\bfe{X}\succeq \bfe{Y}$ and $\bfe{X}\neq \bfe{Y}$ we write $\bfe{X}\succ \bfe{Y}$.
$\bfe{Y}\preceq\bfe{X}$ means that $\bfe{X}\succeq\bfe{Y}$.
\end{definition}
Definition \ref{DefOrderMon} implies that every monomial order on $\Mon(A^m)$ is a well order. Next we give a monomial order $\succeq$ on $\Mon(A)$, we can define two natural orders on $\Mon(A^m)$.
\begin{definition}
Let $\bfe{X}=X\bfe{e}_i$ and $\bfe{Y}=Y\bfe{e}_j\in\Mon(A^m)$.
\begin{enumerate}
\item [\rm(i)] The TOP {\rm(}term over position{\rm)} order is defined by
\quad $\bfe{X}\succeq\bfe{Y}\Longleftrightarrow
\begin{cases} X\succeq Y & \\
\text{or} & \\
X=Y \text{and}& i>j.
\end{cases}$
\item[\rm(ii)] The TOPREV order is defined by\quad
$\bfe{X}\succeq\bfe{Y}\Longleftrightarrow
\begin{cases} X\succeq Y & \\
\text{or} & \\
X=Y \text{and}& i<j.
\end{cases}$
\end{enumerate}
\end{definition}
\begin{remark}\hfill
\begin{enumerate}
\item[\rm (i)] Note that with TOP we have
$\textbf{\emph{e}}_m\succ \textbf{\emph{e}}_{m-1}\succ\cdots\succ\textbf{\emph{e}}_1
$ and $\textbf{\emph{e}}_1\succ \textbf{\emph{e}}_{2}\succ\cdots\succ\textbf{\emph{e}}_m$
for TOPREV.
\item[\rm (ii)] The POT (position over term) and POTREV  orders defined in \cite{Loustaunau} and \cite{Lezama2} for modules over classical polynomial commutative rings are not degree compatible.
\end{enumerate}
\end{remark}
We fix a monomial order on $\Mon(A)$ and a non-zero vector $\textbf{\emph{f}}\in A^m.$ Then we write $\textbf{\emph{f}}$ as a sum of terms in the following form
\[
\textbf{\emph{f}}=c_1\textbf{\emph{X}}_1+\cdots+c_t\textbf{\emph{X}}_t,
\]
where $c_1,\dots ,c_t\in R-0$ and $\textbf{\emph{X}}_1\succ\textbf{\emph{X}}_2\succ\cdots\succ\textbf{\emph{X}}_t$ are monomials of $\Mon(A^m)$.
\begin{definition}\label{DefRepresentMon}
Let $\bfe{f}:=c_1\textbf{\emph{X}}_1+\cdots+c_t\textbf{\emph{X}}_t\in A^{m}$ where $c_1,\dots ,c_t\in R-0,$ $\textbf{\emph{X}}_1\succ\textbf{\emph{X}}_2\succ\cdots\succ\textbf{\emph{X}}_t$ monomials of $\Mon(A^m)$ and $\bfe{X}_i:=x^{\gamma_i}\bfe{e}_{j_i}$ with $\gamma_i\in\mathbb{N}^{n}$. Given $g\in A$, we define
\[
\bfe{f}g:=c_1x^{\gamma_{1}}g\bfe{e}_{j_1}+\cdots+c_tx^{\gamma_{t}}g\bfe{e}_{j_t}
\]
and we view $fg$ as an element of $A^{m}.$
\end{definition}
\begin{remark}
In the notation of Definition \ref{DefRepresentMon}, we have
$\exp(\lm(\bfe{f}x^{\alpha}))=\exp(\lm(\bfe{f}))+\alpha.$ In fact, as $\succ$ is monomial order on $\Mon(A^{m})$, then $\lm(x^{\gamma_1}x^{\alpha})\bfe{e}_{j_1}\succ\lm(x^{\gamma_k}x^{\alpha})\bfe{e}_{j_k}$ for each $2\leq k\leq t,$ thus, $\lm(\bfe{f}x^{\alpha})=\lm(x^{\gamma_1}x^{\alpha})\bfe{e}_{j_1}$ so, $\exp(\lm(\bfe{f}x^{\alpha}))=\gamma_1+\alpha=\exp(\lm(\bfe{f}))+\alpha.$ Hence, $\lc(\bfe{f}x^{\alpha})=c_1c_{\gamma_1,\alpha}=\lc(\bfe{f})c_{\gamma_1,\alpha}$.
\end{remark}
\begin{definition}
Under the notation introduced earlier, say that:
\begin{enumerate}
\itemsep=0.9pt
\item[\rm(i)]$\lt(\bfe{f}):=c_1\bfe{X}_1$ is the leading term of $\bfe{f},$
\item[\rm(ii)]$\lc(\bfe{f}):=c_1$ is the leading coefficient of $\bfe{f},$
\item[\rm(iii)]$\lm(\bfe{f}):=\bfe{X}_1$ is the leading monomial of $\bfe{f}$.
\end{enumerate}
\end{definition}
For $\bfe{\emph{f}}=\bfe{0}$ we define
$\lm(\bfe{0})=\bfe{0},\lc(\bfe{0})=0,\lt(\bfe{0})=\bfe{0}$, and if $\succeq$ is a
monomial order on $\Mon(A^m)$, then we define $\bfe{X}\succ\bfe{0}$ for any $\bfe{X}\in
\Mon(A^m)$. So, we extend $\succeq$ to $\Mon(A^m)\bigcup\{\bfe{0}\}$.

\section{Right reduction in $A^m$}

In this section we present the fundamental topics of reduction theory for right submodules of the free $A$-module $A^m$ when $A$ is a bijective skew $PBW$ extension. This theory was studied in the bijective general case for left modules. Here we adapt the ideas used in \cite{Fajardo3}.

Throughout are assume that $R$ satisfies some natural computational conditions.
\begin{definition}\label{RGSring}
A ring $R$ is right Gröbner soluble {\rm(}$RGS${\rm)} if the following conditions hold:
\begin{enumerate}
\itemsep=0.9pt
\item[\rm (i)]$R$ is right Noetherian.
\item[\rm (ii)]Given $a,r_1,\dots,r_m\in R$ there exists an algorithm which decides whether $a$ is in the right ideal $r_1R+\cdots+r_mR$, and if so, find $b_1,\dots,b_m\in R$ such that
    $a=r_1b_1+\cdots+r_mb_m$.
\item[\rm (iii)]Given $r_1,\dots,r_m\in R$ there exists an algorithm which finds a finite set of generators of the right $R$-module
    \[
    \Syz_R^{r}[r_1\ \cdots \ r_m]:=\{(b_1,\dots,b_m)\in R^m|r_1b_1+\cdots+r_mb_m=0\}.
    \]
\end{enumerate}
\end{definition}
\begin{remark}\label{LGS}
The three conditions {\rm(i) - (iii)}imposed on $R$ are needed in order to guarantee a right Gröbner theory in the
rings of coefficients, in particular, to have an effective solution of the membership problem in
$R$ (see (ii) in Definition \ref{reductionformodules} below). From now on in this paper we
will assume that $A=\sigma(R)\langle x_1,\dots,x_n\rangle$ is a skew $PBW$ extension of $R$, where
$R$ is a $RGS$ ring and $\Mon(A)$ is endowed with some monomial order.
\end{remark}

The reduction process in $A^m$ is defined as follows.
\begin{definition}\label{reductionformodules}
Let $F$ be a finite set of non-zero vectors of $A^m$, and let $\bfe{f},\bfe{h}\in A^m$, we say that $\bfe{f}$ reduces to $\bfe{h}$ by $F$ in one step, denoted $\bfe{f}\xmapsto{\,\, F\,\, } \bfe{h}$, if there exist elements $\bfe{f}_1,\dots,\bfe{f}_t\in F$ and $r_1,\dots,r_t\in R$ such that
\begin{enumerate}
\item[\rm(i)] $\lm(\bfe{f}_i)|\lm(\bfe{f})$, $1\leq i\leq t$, i.e.,
$\ind(\lm(\bfe{f}_{i}))=\ind(\lm(\bfe{f}))$ and there exists $x^{\alpha_{i}}\in \Mon(A)$ such that $\beta_i+\alpha_{i}=\exp(\lm(\bfe{f}))$ with $\beta_i:=\exp(\lm(\bfe{f}_i))$.
\eject
\item[\rm(ii)] $ \lc(\bfe{f})=\lc(\bfe{f}_1)\sigma^{\beta_1}(r_1)c_{\bfe{f}_1,\alpha_1}+
\cdots+\lc(\bfe{f}_t)\sigma^{\alpha_t}(r_t)c_{\bfe{f}_t,\alpha_t}$, where
$c_{\bfe{f}_i,\alpha_i}:=c_{\beta_i,\alpha_i}$.
\item[\rm (iii)]$\bfe{h}=\bfe{f}-\sum_{i=1}^t\bfe{f}_ir_ix^{\alpha_i}$.
\end{enumerate}
We say that $\bfe{f}$ reduces to $\bfe{h}$ by $F$, denoted
$\bfe{f}\xmapsto{\,\, F\,\, }_{+} \bfe{h}$, if and only
if there exist vectors $\bfe{h}_1,\dots ,\bfe{h}_{t-1}\in
A^m$ such that

$$\bfe{f}\ \xmapsto{\ \ \ F\ \ \ }\ \bfe{h}_1\xmapsto{\ \ \ F\ \ \ }\ \bfe{h}_2\ \xmapsto{\ \ \ F\ \ \ }\ \cdots\ \xmapsto{\ \ \ F\ \ \ }\ \bfe{h}_{t-1}\ \xmapsto{\ \ \  F\ \ \ }\ \bfe{h}.$$

\smallskip

\noindent
$\bfe{f}$ is reduced {\rm(}also called minimal{\rm)} with respect to $F$
if $\bfe{f} = \bfe{0}$ or there is no one step
reduction  of $\bfe{f}$ by $F$, i.e., one of the first two
conditions of Definition \ref{reductionformodules} fails.
Otherwise, we will say that $\bfe{f}$ is reducible with respect to $F$.
If $\bfe{f}\xmapsto{\,\, F\,\, }_{+} \bfe{h}$ and
$\bfe{h}$ is reduced with respect to $F$, then we say that $\bfe{h}$
is a remainder for $\bfe{f}$ with respect to $F$.
\end{definition}
\begin{remark}
Related to the previous definition we have the following remarks:
\begin{enumerate}
\item[\rm (i)] By Theorem \ref{coefficientes}, the coefficients $c_{\bfe{f}_i,\alpha_i}$ in the previous definition are unique and satisfy
    \[
    x^{\exp(\lm(\textbf{\emph{f}}_i))}x^{\alpha_i}=
    c_{\textbf{\emph{f}}_i,\alpha_i}
    x^{\exp(\lm(\textbf{\emph{f}}_i))+\alpha_i}+
    p_{\textbf{\emph{f}}_i,\alpha_i},
    \]
    where $p_{\textbf{\emph{f}}_i,\alpha_i}=0$ or
    $\deg(\lm(p_{\textbf{\emph{f}}_i,\alpha_i}))<
    |\exp(\lm(\textbf{\emph{f}}_i))+\alpha_i|$, $1\leq i\leq t$.
\item[\rm (ii)] $\lm(\textbf{\emph{f}})\succ \lm(\textbf{\emph{h}})$ and $\textbf{\emph{f}}-\textbf{\emph{h}}\in \langle F\rangle$, where $\langle F\rangle$ is the right submodule of $A^m$ generated by $F$.
\item[\rm (iii)] The remainder of $\textbf{\emph{f}}$ is not unique.
\item[\rm (vi)] By definition we will assume that $\textbf{0}\xmapsto{F}\textbf{0}$.
\item[\rm (v)] \vspace*{-3mm}
$$ \lt(\textbf{\emph{f}})= \sum_{i=1}^{t}\lt(\lt(\textbf{\emph{f}}_i)r_i x^{\alpha_{i}}),$$
\end{enumerate}
\end{remark}

From the reduction relation we obtain the following interesting properties.

\begin{proposition}\label{xtetafformodules}
Assume that $A$ is a bijective skew $PBW$ extension. Let $\bfe{f},\bfe{h}\in A^{m}$, $\theta \in \mathbb{N}^n$ and $F=\{\bfe{f}_1,\dots,\bfe{f}_t\}$ be a finite set of non-zero vectors of $A^{m}$.

\begin{enumerate}
\item[\rm{(i)}]If $\bfe{f}\xmapsto{\,\, F\,\, }\bfe{h}$, then there exists $\bfe{p}\in
A^{m}$ with $\bfe{p}=\bfe{0}$ or $\lm(\bfe{f}x^{\theta})\succ \lm\bfe{(p})$ such that
$\bfe{f}x^{\theta}+\bfe{p}\xmapsto{\,\, F\,\, } \bfe{h}x^{\theta}$.
\item[\rm{(ii)}]If $\bfe{f}\xmapsto{\,\, F\,\, }_+\bfe{h}$ and $\bfe{p}\in A$ is such that $\bfe{p}=\bfe{0}$ or $\lm(\bfe{h})\succ \lm(\bfe{p})$, then
$\bfe{f}+\bfe{p}\xmapsto{\,\, F\,\, }_+ \bfe{h}+\bfe{p}$.
\item[\rm{(iii)}]If $\bfe{f}\xmapsto{\,\, F\,\, }_+ \bfe{h}$, then there
exists $\bfe{p}\in A^{m}$ with $\bfe{p}=\bfe{0}$ or $\lm(\bfe{f}x^{\theta})\succ \lm(\bfe{p})$ such
that $\bfe{f}x^{\theta}+\bfe{p}\xmapsto{\,\, F\,\, }_+ x^{\theta}\bfe{h}$.
\item[\rm{(iv)}]If $\bfe{f}\xmapsto{\,\, F\,\, }_+ \bfe{0}$, then there
exists $\bfe{p}\in A^{m}$ with $\bfe{p}=\bfe{0}$ or $\lm(\bfe{f}x^{\theta})\succ \lm(\bfe{p})$ such
that $\bfe{f}x^{\theta}+\bfe{p}\xmapsto{\,\, F\,\, }_+ \bfe{0}$.
\end{enumerate}\smallskip
\end{proposition}

\begin{proof}

\vspace*{-7mm}
\begin{enumerate}
\item[\rm (i)] If $\bfe{f}=\bfe{0}$, then $\bfe{h}=\bfe{0}=\bfe{p}$. Let $\bfe{f}\neq\bfe{0}$ and $\lm(\bfe{f}):=x^{\lambda}$; then there exist $\bfe{f}_1,\dots,\bfe{f}_t\in F$ and $r_1\dots,r_t\in R$ such that $\lm(\bfe{f}_i)|\lm(\bfe{f})$, for $1\leq i\leq t,$ i.e., $\ind(\lm(\bfe{f}_i))=\ind(\lm(\bfe{f}))$ and there exists $x^{\alpha_i}\in \Mon(A)$ such that
    $\lambda=\alpha_i+\exp(\lm(\bfe{f}_i))$. Moreover,
    \[
    \lc(\bfe{f})=\lc(\bfe{f}_1)\sigma^{\beta_1}(r_1)c_{\beta_1,\alpha_1}+
    \cdots +\lc(\bfe{f}_t)\sigma^{\beta_t}(r_t)c_{\beta_t,\alpha_t}
    \]
    with $\beta_i:=\exp(\lm(\bfe{f}_i))$ and $\bfe{h}=\bfe{f}-\sum_{i=1}^t\bfe{f}_ir_ix^{\alpha_i}.$
    We note that $\ind(\lm(\bfe{f}))=\ind(\lm(\bfe{f}x^{\theta}))$ and $\exp(\bfe{f}x^{\theta})=\theta+\lambda$, so
\[
\lm(\bfe{f}_i)|\lm(\bfe{f}x^{\theta}),\text{ with } \theta+\lambda=(\theta+\alpha_i)+\beta_i;
\]
we observe that
\[
\lc(\bfe{f}x^{\theta})=\lc(\bfe{f})c_{\lambda,\theta}=
\sum_{i=1}^t \lc(\bfe{f}_i)\sigma^{\beta_i}(r_i)c_{\beta_i,\alpha_i}c_{\lambda,\theta}.
\]
Hence Remark \ref{identities} yields:\vspace*{-1mm}
\begin{align*}
\lc(\bfe{f}x^{\theta}) &= \sum_{i=1}^t \lc(\bfe{f}_i)\sigma^{\beta_i}(r_i)c_{\beta_i,\alpha_i}
c_{\alpha_i+\beta_i,\theta}\\
&=\sum_{i=1}^tlc(\bfe{f}_i)\sigma^{\beta_i}(r_i)
\sigma^{\beta_i}(c_{\alpha_i,\theta})c_{\beta_i,\alpha_i+\theta}\\
&=\sum_{i=1}^tlc(\bfe{f}_i)\sigma^{\beta_i}
(r_ic_{\alpha_i,\theta})c_{\beta_i,\alpha_i+\theta}\\
&=\sum_{i=1}^tlc(\bfe{f}_i)\sigma^{\beta_i}(r_i')c_{\beta_i,\alpha_i+\theta},
\end{align*}
where $r_i':=r_ic_{\alpha_i,\theta}$. Moreover,
\begin{align*}
\bfe{h}x^{\theta}&=\bfe{f}x^{\theta}-\sum_{i=1}^t
\bfe{f}_ir_ix^{\alpha_i}x^{\theta}\\
&=\bfe{f}x^{\theta}-\sum_{i=1}^t\bfe{f}_ir_ic_{\alpha_i,\theta}
x^{\alpha_i+\theta}+\bfe{p}\\
&=\bfe{f}x^{\theta}+\bfe{p}-\sum_{i=1}^t\bfe{f}_ir_i'x^{\alpha_i+\theta}
\end{align*}
where $\bfe{p}:=\sum_{i=1}^t(-\bfe{f}_i)r_ip_{\alpha_i,\theta}$; note that $\bfe{p}=\bfe{0}$ or $\deg(\bfe{p})<|\theta+\alpha_i+\beta_i|=|\theta+\lambda|=
deg(\bfe{f}x^{\theta})$, so $\lm(\bfe{f}x^{\theta})\succ \lm(\bfe{p})$. Moreover, $\lm(\bfe{f}x^{\theta}+\bfe{p})=\lm(\bfe{f}x^{\theta})$ and
$\lc(\bfe{f}x^{\theta}+\bfe{p})=\lc(\bfe{f}x^{\theta})$, so by the previous discussion
$x^{\theta}\bfe{f}+\bfe{p}\xmapsto{\,\, F\,\, } x^{\theta}\bfe{h}$.
\item[\rm (ii)] Let
\begin{equation}\label{f+p}
\bfe{f}\ \xmapsto{\ \ \ F\ \ \ }\ \bfe{h}_1\ \xmapsto{\ \ \ F\ \ \ }\ \bfe{h}_2\ \xmapsto{\ \ \ F\ \ \ }\ \cdots\ \xmapsto{\ \ \ F\ \ \ }\ \bfe{h}_{t-1}\ \xmapsto{\ \ \ F\ \ \ }\ \bfe{h}_t:=\bfe{h}.
\end{equation}
We start with $f\xmapsto{F} \bfe{h}_1.$ If $\bfe{f}=\bfe{0}$, then $\bfe{h}_1=\bfe{0}=\bfe{p}$ and there is nothing to prove. Let $\bfe{f}\neq \bfe{0}.$ If $\bfe{h}_1=\bfe{0}$ then $\bfe{p}=\bfe{0}$ and hence $\lm(\bfe{f})\succ \lm(\bfe{p})$; if $\bfe{h}_1\neq \bfe{0} $, then $\lm(\bfe{f})\succ \lm(\bfe{h}_1)\succ \lm(\bfe{p})$, and hence $\lm(\bfe{f}+\bfe{p})=\lm(\bfe{f})$, $\lc(\bfe{f}+\bfe{p})=\lc(\bfe{f}).$ Now, as in the proof of the first part of (i), we obtain $\bfe{h}_1+\bfe{p}=\bfe{f}+\bfe{p}-\sum_{i=1}^t\bfe{f}_ir_ix^{\alpha_i}.$ Since $\lm(\bfe{f}+\bfe{p})=\lm(\bfe{f})$ and $\lc(\bfe{f}+\bfe{p})=\lc(\bfe{f})$, then $\bfe{f}+\bfe{p}\xmapsto{\,\, F\,\, } \bfe{h}_1+\bfe{p}$. Since $\lm(\bfe{h}_i)\succ \lm(\bfe{p})$ we can repeat this procedure for $\bfe{h}_i\xmapsto{\,\, F\,\, } \bfe{h}_{i+1}$ with $1\leq i\leq t-1$. This completes the proof of (ii).
\item[\rm (iii)] By (i) and (\ref{f+p}), there exists $\bfe{p}_1\in A^m$ with $\bfe{p}_1=\bfe{0}$ or $\lm(\bfe{f}x^{\theta})\succ \lm(\bfe{p}_1)$ such that $\bfe{f}x^{\theta}+\bfe{p}_1\xmapsto{\,\, F\,\, } \bfe{h}_1x^{\theta}.$ Moreover there exists $\bfe{p}_2\in A^m$ with $\bfe{p}_2=0$ or $\lm(\bfe{h}_1x^{\theta})\succ \lm(\bfe{p}_2)$
such that $\bfe{h}_1x^{\theta}+\bfe{p}_2\xmapsto{\,\, F\,\, }\bfe{h}_2x^{\theta}.$ Hence, in view of (ii), we obtain $\bfe{f}x^{\theta}+\bfe{p}_1+\bfe{p}_2\xmapsto{\,\, F\,\, }\bfe{h}_1x^{\theta}+\bfe{p}_2\xmapsto{\,\, F\,\, } \bfe{h}_2x^{\theta}$, so the element $\bfe{p}'':=\bfe{p}_1+\bfe{p}_2\in A^m$ is such that\vspace*{-1mm}
\begin{center}
$\bfe{f}x^{\theta}+\bfe{p}''\xmapsto{\,\, F\,\, }_+ \bfe{h}_2x^{\theta}$,
\end{center}
with $\bfe{p}''=0$ or $\lm(\bfe{f}x^{\theta})\succ \lm(\bfe{p}''),$ because we have $\lm(\bfe{f}x^{\theta})\succ \lm(\bfe{p}_1)$ and $\lm\bfe{f}(x^{\theta})\succ \lm(\bfe{p}_2)$. By induction on $t$ we find $\bfe{p}'\in A^m$ such that\vspace*{-1mm}
\begin{center}
$\bfe{f}x^{\theta}+\bfe{p}'\xmapsto{\,\, F\,\, }_+h_{t-1}x^{\theta}$,
\end{center}
with $\bfe{p}'=0$ or $\lm(\bfe{f}x^{\theta})\succ \lm(\bfe{p}')$. By (i) there exists $\bfe{p}_t\in A^m$ such that $\bfe{h}_{t-1}x^{\theta}+\bfe{p}_t\xmapsto{\,\, F\,\,}hx^{\theta}$, with $\bfe{p}_t=\bfe{0}$ or $\lm(\bfe{h}_{t-1}x^{\theta})\succ \lm(\bfe{p}_t)$. By (ii), $\bfe{f}x^{\theta}+\bfe{p}'+\bfe{p}_t\xmapsto{\,\, F\,\, }_+
\bfe{h}_{t-1}x^{\theta}+\bfe{p}_t\xmapsto{\,\, F\,\, }\bfe{h}x^{\theta}$. Thus,
\begin{center}
$\bfe{f}x^{\theta}+\bfe{p}\xmapsto{\,\, F\,\, }_+\bfe{h}x^{\theta}$,
\end{center}
with $\bfe{p}:=\bfe{p}'+\bfe{p}_t=0$ or $\lm(\bfe{f}x^{\theta})\succ \lm(\bfe{p})$ since $\lm(\bfe{f}x^{\theta})\succ \lm(\bfe{p}')$ and $\lm(\bfe{f}x^{\theta})\succ \lm(\bfe{p}_t)$.
\item[\rm (iv)] This is a direct consequence of (iii) taking $\bfe{h}=\bfe{0}$.
\end{enumerate}

\vspace*{-9mm}
\end{proof}\smallskip

\begin{definition}\label{definicionPsi}
Let $A:=\sigma(R)\langle x_1,\ldots,x_n\rangle$ a bijective skew $PBW$ extension. Let $\theta_1,\theta_2\in\mathbb{N}^{n}$. We define the following automorphism over $R$, $\psi_{\theta_1,\theta_2}:R\to R$ that assigns to each $r\in R$.
\[
\psi_{\theta_1,\theta_2}(r):=\sigma^{\theta_1+\theta_2}
(\sigma^{-\theta_2}(r)).
\]
\end{definition}

\begin{remark}\label{defPsi}\hfill
\begin{enumerate}
\item[\rm (i)] The inverse function of $\psi$ is given by $\psi_{\theta_1,\theta_2}^{-1}(r)=\sigma^{\theta_2}
    \sigma^{-(\theta_1+\theta_2)}(r).$
\item[\rm (ii)] Let $A:=\sigma(R)\langle x_1,\ldots,x_n\rangle$ a bijective skew $PBW$ extension. For $\alpha,\beta,\gamma\in\mathbb{N}^{n}$ and $r\in R$, using the identities of Remark \ref{identities}, we obtain
    \begin{align}\label{id1Psi}
    \sigma^{\beta}(r)c_{\beta,\alpha} &= c_{\beta,\alpha}\psi_{\beta,\alpha}(r)\\\label{id2Psi}
    c_{\beta,\alpha}r &=\sigma^{\beta}(\psi^{-1}_{\beta,\alpha}(r))c_{\beta,\alpha}.
    \end{align}
    \eject
    Moreover, we have
\begin{equation}\label{identitypsi}
\begin{split}
c_{\beta,\alpha}rc_{\beta+\alpha,\gamma}
&= \sigma^{\beta}(\psi^{-1}_{\beta,\alpha}(r))
c_{\beta,\alpha}c_{\alpha+\beta,\gamma}\\
&=\sigma^{\beta}(\psi^{-1}_{\beta,\alpha}(r))
\sigma^{\beta}(c_{\alpha,\gamma})c_{\beta,\alpha+\gamma}\\
&=\sigma^{\beta}(\psi^{-1}_{\beta,\alpha}(r)c_{\alpha,\gamma})
c_{\beta,\alpha+\gamma}\\
&=c_{\beta,\alpha+\gamma}
\psi_{\beta,\alpha+\gamma}(\psi^{-1}_{\beta,\alpha}(r)c_{\alpha,\gamma}).
\end{split}
\end{equation}
\item[\rm (iii)] Under the notation used in proof of Proposition \ref{xtetafformodules} {\rm (i)}; $s_1,\ldots,s_t$ are solutions of the equation
\[
\lc(\bfe{h})=\sum_{i=1}^{t}\lc(\bfe{f}_i)c_{\beta_i,\alpha_i}s_i,
\]
if and only if, $r_i=\psi_{\alpha_i,\beta_i}^{-1}(s_i)$ for $i=1,\ldots,t,$ are solutions of the equation
\[
\lc(\bfe{h})=\sum_{i=1}^{t}\lc(\bfe{f}_i)c_{\beta_i,\alpha_i}
\psi_{\beta_i,\alpha_i}(r_i).
\]
\end{enumerate}
\end{remark}

\noindent The following theorem is a theoretical support of the right Division Algorithm (Algorithm \ref{algoDivisioninAm}) for bijective skew $PBW$ extensions.

\begin{algorithm}[h!]\small
\caption{Right division algorithm in $A^{m}$}
\KwIn{$\bfe{f},\bfe{f}_1,\dots ,\bfe{f}_t\in A^{m}\text{ with }\bfe{f}_j\neq \bfe{0}\, (1\leq j\leq t)$}
\KwOut{$q_1,\dots ,q_t\in A,\,\bfe{h}\in A^{m}\,\,\text{with}\,\, \bfe{f}=\bfe{f}_1q_1+\cdots
+\bfe{f}_tq_t+\bfe{h}$, $\bfe{h}$ reduce with respect to $\{\bfe{f}_1,\dots ,\bfe{f}_t\}$ and
$\lm(\bfe{f})=\max\{\lm(\lm(\bfe{f}_1)\lm(q_1)),\dots ,\lm(\lm(\bfe{f}_t)\lm(q_t)),\lm(\bfe{h})\}$}
\Ini $q_1\gets0,q_2\gets0,\dots ,q_t\gets0,\bfe{h}\gets \bfe{f}$\;
\While{$\bfe{h}\neq \bfe{0}$ and there exists $j$ such that $\lm(\bfe{f}_j)$ divides $\lm(\bfe{h})$}{
  $J\gets\{j\mid \lm(\bfe{f}_j)$ divides $\lm(\bfe{h})\}$;\\
  \For{$i\in J$}{
     $\beta_j\gets\exp(\lm(\bfe{f}_j))$;\\
     $\alpha_j\gets\exp(\lm(\bfe{h}))-\beta_j$;
  }
  \eIf{the equation $\lc(\bfe{h})=\sum_{j\in J}\lc(\bfe{f}_j)c_{\bfe{f}_j,\alpha_j}s_j$ is soluble}{
     Calculate one solution $(s_j)_{j\in J}$;\\
     \For{$j\in J$}{
        $r_j\gets \psi_{\beta_j,\alpha_j}^{-1}(s_j)$;\\
        $q_j\gets q_j+r_jx^{\alpha_j}$;\\
        $\bfe{h}\gets \bfe{h}-\bfe{f}_jr_jx^{\alpha_j}$;
     }
  }{
    \textbf{Break;}
  }
}
\label{algoDivisioninAm}
\end{algorithm}

\begin{theorem}\label{algdivformodules}
Let $F=\{\bfe{f}_1,\dots ,\bfe{f}_t\}$ be a set of non-zero vectors of $A^m$ and $\bfe{f}\in A^m$, then the right division algorithm (Algorithm \ref{algoDivisioninAm}) produces polynomials $q_1,\dots ,q_t\in A$ and a reduced vector $\bfe{h}\in A^m$ with respect to $F$ such that $\bfe{f}\xmapsto{\,\, F\,\, }_{+} \bfe{h}$ and
\begin{center}
$\bfe{f}=\bfe{f}_1q_1+\cdots +\bfe{f}_tq_t+\bfe{h}$
\end{center}
with
\begin{center}
$\lm(\bfe{f})=\max\{\lm(\lm(\bfe{f}_1)\lm(q_1)),\dots,
\lm(\lm(\bfe{f}_t)\lm(q_t)),\lm(\bfe{h})\}$.
\end{center}

\end{theorem}
\begin{proof}
We first note that Algorithm \ref{algoDivisioninAm} is the iteration of the reduction process. If $\textbf{\emph{f}}$ is reduced with respect to $F:=\{\textbf{\emph{f}}_1,\dots ,\textbf{\emph{f}}_t\}$, then $\textbf{\emph{h}}=\textbf{\emph{f}}, q_1=\cdots =q_t=0$ and
$\lm(\textbf{\emph{f}})=\lm(\textbf{\emph{h}})$. If $\textbf{\emph{f}}$ is not reduced, then we make the first reduction, $\textbf{\emph{f}}\xmapsto{\,\, F\,\, } \textbf{\emph{h}}_1$, with $\textbf{\emph{f}}=\sum_{j\in J_1}\textbf{\emph{f}}_jr_{j1}x^{\alpha_j}+\textbf{\emph{h}}_1$, with $J_1:=\{j\mid \lm(\textbf{\emph{f}}_j)$ divides $\lm(\textbf{\emph{f}})\}$ and $r_{j1}\in R$. If
$\textbf{\emph{h}}_1$ is reduced with respect to $F$, then the cycle \textbf{While} ends and we obtain $q_j=r_{j1}x^{\alpha_j}$ for $j\in J_1$ and $q_j=0$ for $j\notin J_1$. Moreover,
$\lm(\textbf{\emph{f}})\succ \lm(\textbf{\emph{h}}_1)$ and $\lm(\textbf{\emph{f}})=\lm(\lm(\textbf{\emph{f}}_j)\lm(q_j))$ for $j\in J_1$ such that $r_{j1}\neq 0$, hence, $\lm(\textbf{\emph{f}})=\max_{1\leq j\leq t }\{\lm(\lm(\textbf{\emph{f}}_j))\lm(q_j),\lm(\textbf{\emph{h}}_1)\}$. If $\textbf{\emph{h}}_1$ is not reduced, so we make the second reduction with respect to $F$,
$\textbf{\emph{h}}_1\xmapsto{\,\, F\,\, } \textbf{\emph{h}}_2$, with $\textbf{\emph{h}}_1=\sum_{j\in J_2}\textbf{\emph{f}}_jr_{j2}x^{\alpha_j}+\textbf{\emph{h}}_2$,
$J_2:=\{j\mid \lm(\textbf{\emph{f}}_j)$ divides $\lm(\textbf{\emph{h}}_1)\}$ and $r_{j2}\in R$. We have
\begin{center}
$\textbf{\emph{f}}=\sum_{j\in J_1}\textbf{\emph{f}}_jr_{j1}x^{\alpha_j}+\sum_{j\in
J_2}\textbf{\emph{f}}_jr_{j2}x^{\alpha_j}+\textbf{\emph{h}}_2$
\end{center}
If $\textbf{\emph{h}}_2$ is reduced with respect to $F$ the procedure ends and we get $q_j=q_j$ for $j\notin J_2$ and $q_j=q_j+r_{j2}x^{\alpha_j}$ for $j\in J_2$. Since
$\lm(\textbf{\emph{f}})\succ \lm(\textbf{\emph{h}}_1)\succ \lm(\textbf{\emph{h}}_2)$, then the algorithm produces polynomials $q_j$ with monomials ordered according to the monomial
order fixed, and again we have $\lm(\textbf{\emph{f}})=\max_{1\leq j\leq t
}\{\lm(\lm(q_j)\lm(\textbf{\emph{f}}_j)),\lm(\textbf{\emph{h}}_2)\}$. If we continue this way, the algorithm ends since $\Mon(A^m)$ is well ordered.
\end{proof}

\section{Gröbner bases for right submodules of $A^m$}

\noindent
In this section we present the general theory of Gröbner bases for right submodules of $A^m$, $m\geq 1$, where $A=\sigma(R)\langle x_1,\dots,x_n\rangle$ is a bijective skew $PBW$ extension of $R$, with $R$ a $RGS$ ring (see Definition \ref{RGSring}) and $\Mon(A)$ endowed with some monomial order (see Definition \ref{monomialorder}). $A^m$ is the right free $A$-module of column vectors of length $m\geq 1$; since $A$ is a right Noetherian ring, then $A$ is an $IBN$ ring (Invariant Basis Number, see \cite{Lezama6}), and hence, all bases of the free module $A^m$ have $m$ elements. Note moreover that $A^m$ is right Noetherian, and hence, any submodule of $A^m$ is finitely generated.

The plan is to define and calculate Gröbner bases for right submodules of $A^m$, we will present some equivalent conditions in order to define right Gröbner bases, and finally, we will compute right Gröbner bases using a procedure similar to right Buchberger's algorithm over bijective skew $PBW$ extensions.
This theory was studied in the bijective general case for left modules. Here we adapt the ideas and technique used in \cite{Fajardo3}.

\medskip
Our next purpose is to define Gröbner bases for right submodules of $A^m$.

\begin{definition}
Let $M\neq 0$ be a right submodule of $A^m$ and let $G$ be a non-empty finite subset of non-zero vectors of $M$, we say that $G$ is a Gröbner basis for $M$ if each element $0\neq \bfe{f}\in M$ is reducible with respect to $G$. We will say that $\{\bfe{0}\}$ is a Gröbner basis for $M=0$.
\end{definition}

\begin{theorem}\label{teogrobnersigmapbwformodules}
Let $M\neq 0$ be a right submodule of the free $A$-module $A^{m}$ and let $G$ be a finite subset of non-zero vectors of $M$. Then the following conditions are equivalent:
\begin{enumerate}
\item[\rm(i)]$G$ is a Gröbner basis for $M$.
\item[\rm(ii)]For any vector $\bfe{f}\in A^{m}$,\vspace*{-2mm}
\begin{center}
$\bfe{f}\in M$ if and only if $\bfe{f}\xmapsto{\,\, G\,\, }_{+}\bfe{0}$.
\end{center}
\item[\rm(iii)]For any $\bfe{0}\neq \bfe{f}\in M$ there exist $\bfe{g}_1,\dots ,\bfe{g}_t\in G$ such that $\lm(\bfe{g}_j)|\lm(\bfe{f})$, $1\leq j\leq t$, {\rm(}i.e., $\ind(\lm(\bfe{g}_j))=\ind(\lm(\bfe{f}))$ and there exist $\alpha_j\in \mathbb{N}^n$ such that $\exp(\lm(\bfe{g}_j))+\alpha_j=\exp(\lm(\bfe{f}))${\rm)} and
    \[
    \lc(\bfe{f})\in \{\lc(\bfe{g}_1)c_{\bfe{g}_1,\alpha_1},\dots,\lc(\bfe{g}_t)
    c_{\bfe{g}_t,\alpha_t}\rangle.
    \]
\item[\rm(iv)] For $\alpha \in \mathbb{N}^n$ and $1\leq v\leq m,$ let $\{\alpha ,I\rangle_v$ be the right ideal of $R$ defined by
    \[
    \{\alpha ,M\rangle_v:=\{\{\lc(\bfe{f})\mid\bfe{f}\in M,\,\ind(\lm\bfe{f})=v,\,\exp(\lm(\bfe{f}))=\alpha\}\rangle.
    \]
    Then, $\{\alpha ,I\rangle_v=J_v$, with
    \begin{center}
    $J_v:=\{\{\lc(\bfe{g})c_{\bfe{g},\beta}\mid \bfe{g}\in G,\, \ind(\lm\bfe{g})=v\, \text{ and }\ \exp(\lm(\bfe{g}))+\beta =\alpha\}\rangle$.
    \end{center}
\end{enumerate}
\end{theorem}
\begin{proof}
(i) $\Rightarrow$ (ii): Let $\bfe{f}\in M$, if $\bfe{f}=\bfe{0}$, then by definition $\bfe{f}\xmapsto{\,\, G\,\, }_{+}\bfe{0}$. If $\bfe{f}\neq\bfe{0}$, then there exists $\bfe{h}_1\in A^{m}$ such that $\bfe{f}\xmapsto{\,\, G\,\, }\bfe{h}_1$, with
$\lm(\bfe{f})\succ\lm(\bfe{h}_1)$ and $\bfe{f}-\bfe{h}_1\in \{G\rangle\subseteq M$, hence $\bfe{h}_1\in M$; if $\bfe{h}_1=\bfe{0}$, so we end. If $\bfe{h}_1\neq \bfe{0}$, then we can repeat this reasoning for $\bfe{h}_1$, and since $\Mon(A^{m})$ is well ordered, therefore $\bfe{f}\xmapsto{\,\, G\,\, }_{+}\bfe{0}$.

Conversely, if $f\xmapsto{\,\, G\,\, }_{+} 0$, then by the Theorem \ref{algdivformodules}, there exist $\bfe{g}_1,\dots,\bfe{g}_t\in G$ and $q_1,\dots,q_t\in A$ such that $\bfe{f}=\bfe{g}_1q_1+\cdots +\bfe{g}_tq_t$, i.e., $\bfe{f}\in M$.

\smallskip
(ii) $\Rightarrow$ (i): evident.

\smallskip
(i) $\Leftrightarrow$ (iii): this is a direct consequence of Definition \ref{reductionformodules} and the equation (\ref{id1Psi}).

\smallskip
(iii) $\Rightarrow$ (iv) Since $R$ is a right Noetherian ring, there exist $r_1,\dots ,r_s\in R$, $\bfe{f}_1,\dots,\bfe{f}_n\in M$ such that $\{\alpha ,M\rangle_v=\{r_1,\dots, r_s\rangle$, $\ind(\lm(\bfe{f}_i))=v$, $\lm(f_i)=x^{\alpha}$, $1\leq i\leq n$, with $\{r_1,\dots, r_s\rangle\subseteq $ $\{\lc(f_1),\dots,\lc(f_n)\rangle$, then $\{\lc(\bfe{f}_1),\dots,\lc(\bfe{f}_n)\rangle=\{\alpha,M\rangle_v$. Let $r\in \{\alpha ,M\rangle_v$, there exist $a_1,\dots,a_n\in R$ such that $r=\lc(\bfe{f}_1)a_1+\cdots+\lc(\bfe{f}_n)a_n$; by (iii), for each $i$ there exist $\bfe{g}_{1i},\dots,\bfe{g}_{t_ii}\in G$ and $b_{ji}\in R$ such that
\[
\lc(\bfe{f}_i)=\lc(\bfe{g}_{1i})c_{\bfe{g}_{1i},\alpha_{1i}}b_{1i}+\cdots +\lc(\bfe{g}_{t_ii})c_{\bfe{g}_{t_ii},\alpha_{t_ii}}b_{t_ii},
\]
with $v=\ind(\lm\bfe{f}_i)=\ind(\lm(\bfe{g}_{ji}))$ and $\alpha=\exp(\lm(\bfe{f}_i))=\alpha_{ji}+\exp(\lm(\bfe{g}_{ji}))$, thus $\{\alpha ,M\rangle_v\subseteq J_v$.

Conversely, if $r\in J_v$, then $r=\lc(\bfe{g}_{1})c_{\bfe{g}_{1},\beta_{1}}b_{1}+\cdots
+\lc(\bfe{g}_{t})c_{\bfe{g}_{t},\beta_{t}}b_{t}$, with $b_i\in R$, $\beta_i\in \mathbb{N}^n$, $\bfe{g}_i\in G$ such that $\ind(\lm(\bfe{g}_i))=v$, $\beta_i+\exp(\lm(\bfe{g}_i))=\alpha$ for any $1\leq i\leq t$.

Note that $\bfe{g}_ix^{\beta_i}\in M$, $\ind(\bfe{g}_ix^{\beta_i})=v$, $\exp(\lm(\bfe{g}_ix^{\beta_i}))=\alpha$, $\lc(\bfe{g}_ix^{\beta_i})=\lc(\bfe{g}_{i})c_{\bfe{g}_i,\beta_{i}}$, for $1\leq i\leq t$, thus $r=\lc(\bfe{g}_1x^{\beta_1})b_1+\cdots+\lc(\bfe{g}_tx^{\beta_t})b_t,$ i.e., $r\in \{\alpha ,M\rangle_v$.

\smallskip
(iv)$\Rightarrow$ (iii): let $\bfe{0}\neq \bfe{f}\in M$ and let $\alpha=\exp(\lm(\bfe{f}))$ and
$v=\ind(\lm(\bfe{f}))$, then $\lc(\bfe{f})\in\{\alpha,M\rangle_v$; by (iv) $\lc(\bfe{f})=\lc(\bfe{g}_{1})c_{\bfe{g}_1,\beta_{1}}b_{1}+\cdots
+\lc(\bfe{g}_{t})c_{\bfe{g}_t,\beta_{t}}b_{t}$, with $b_i\in R$, $\beta_i\in \mathbb{N}^n$, $\bfe{g}_i\in G$ such that $=\ind(\lm(\bfe{g}_i))=v$ and $\beta_i+\exp(\lm(\bfe{g}_i))=\alpha$ for any $1\leq i\leq t$. From this we conclude that, $\lm(\bfe{g}_i)|\lm(\bfe{f}).$
\end{proof}

Some useful consequences of Theorem \ref{teogrobnersigmapbwformodules} are the following results.
\begin{corollary}\label{153}
Let $M\neq 0$ be a right submodule of $A^{m}$. Then,
\begin{enumerate}
\item[\rm(i)] If $G$ is a Gröbner basis for $M$, then $M=\langle G\rangle$.
\item[\rm(ii)] Let $G$ be a Gröbner basis for $M.$ If $\bfe{f}\in M$ and
$\bfe{f}\xmapsto{\,\, G\,\, }_{+}\bfe{h}$, with $\bfe{h}$ reduced, then $\bfe{h}=\bfe{0}$.
\item[\rm(iii)] Let $G=\{\bfe{g}_1,\dots,\bfe{g}_t\}$ be a set of non-zero vectors of $M$ with $\lc(\bfe{g}_i)=1$ for each $1\leq i\leq t.$ If given $\bfe{0}\neq\bfe{r}\in M$ there exists $i$ such that $\lm(\bfe{g}_i)|\lm(\bfe{r}),$ then $G$ is a Gröbner basis of $M.$
\end{enumerate}
\end{corollary}

\begin{proof}

\vspace*{-7mm}
\begin{enumerate}
\item[\rm (i)] Apply Theorem \ref{teogrobnersigmapbwformodules}.
\item[\rm (ii)] Assume that $\bfe{f}\in M$ and $\bfe{f}\xmapsto{\,\, G\,\, }_{+}\bfe{h}$, with $\bfe{h}$ reduced. Since $\bfe{f}-\bfe{h}\in\langle G\rangle=M$, then $\bfe{h}\in M$; if $\bfe{h}\neq\bfe{0}$ then $\bfe{h}$ can be reduced by $G$, but this is not possible since $\bfe{h}$ is reduced.
\item[\rm (iii)] Assume that $\textbf{\emph{f}}\in A^m.$ By Theorem \ref{algdivformodules} there exists $\textbf{\emph{r}}$ reduced such that $\textbf{\emph{f}}\xmapsto{\,\, G\,\, }_{+}
\textbf{\emph{r}}$. If $\textbf{\emph{f}}\in M$ then $\textbf{\emph{r}}\in M$; if $\textbf{\emph{r}}\neq \textbf{0}$, then by hypothesis there exists $\textbf{\emph{g}}_i\in G$ such that $lm(\textbf{\emph{g}}_i)$ divides $lm(\textbf{\emph{r}})$, thus, since $lc(\textbf{\emph{g}}_i)=1$, then $\textbf{\emph{r}}$ is reducible, which is a contradiction and therefore, $\textbf{\emph{f}}\xmapsto{\,\, G\,\, }_{+} \textbf{0}$. On the other hand, if $\textbf{\emph{f}}\xmapsto{\,\, G\,\, }_{+} \textbf{0}$, then $\textbf{\emph{f}}\in M$. Now
Theorem \ref{teogrobnersigmapbwformodules} {\rm (ii)} implies that $G$ is Gröbner basis of~$M.$
\end{enumerate}

\vspace*{-7mm}
\end{proof}

\begin{corollary}\label{reducedbasis}
Let $G$ be a Gröbner basis for a right submodule $M$ of $A^{m}$. Given $\bfe{g}\in G$, if $\bfe{g}$ is reducible with respect to $G'= G-\{\bfe{g}\}$, then $G'$ is a Gröbner basis for $M$.
\end{corollary}
\begin{proof}
According to Theorem \ref{teogrobnersigmapbwformodules}, it is enough to show that every $\bfe{f}\in M$ is reducible with respect to $G'$. Let $\bfe{f}$ be a nonzero vector in $M$; since $G$ is a Gröbner basis for $M$, $\bfe{f}$ is reducible with respect to $G$ and  there exist elements $\bfe{g}_1,\ldots,\bfe{g}_t\in G$ satisfying the conditions (i), (ii) and (iii) in the Definition \ref{reductionformodules}. If $\bfe{g}\neq \bfe{g}_i$ for each $1\leq i\leq t$, then we finished. Suppose that $\bfe{g}=\bfe{g}_j$ for some $j\in\{1,\ldots,t\}$ and let
$\beta_i=\exp(\lm(\bfe{g}_i))$ for $i\neq j$, $\beta=\exp(\lm(\bfe{g}))$, and $\alpha_i,\alpha\in\mathbb{N}^n$ such that $\alpha_i+\beta_i=\exp(\lm(\bfe{f}))=\alpha+\beta$. Thus,
\[
\lc(\bfe{f})=
\lc(\bfe{g}_1)c_{\beta_1,\alpha_1}r_1+\cdots+\lc(\bfe{g})c_{\beta,\alpha}r_j+
\cdots+\lc(\bfe{g}_t)c_{\beta_t,\alpha_t}r_t.
\]
On the other hand, since $\bfe{g}$ is reducible with respect to $G'$, there exist $\bfe{g}_1',\ldots,\bfe{g}_s'\in G'$ such
that $\lm(\bfe{g}_l')\mid \lm(\bfe{g})$ and $\lc(\bfe{g})=\sum_{k=1}^{s}\lc(\bfe{g}_k')c_{\beta_k',\alpha_k'}r_k'$, where
$\beta_k'=\exp(\lm(\bfe{g}_k'))$, $\alpha_k'\in\mathbb{N}^n$ and $\alpha_k'+\beta_k'=\exp(\lm(\bfe{g}))=\beta$. Thus,
$\lm(\bfe{g}_k')\mid \lm(\bfe{f})$ for $1\leq i \leq s$; moreover, using the equation (\ref{identitypsi}) of Remark \ref{defPsi}, we have
\[
c_{\beta_k',\alpha_k'}r_k'c_{\beta,\alpha}=c_{\beta_k',\alpha_k'}r_k'
c_{\beta_k'+\alpha_k',\alpha}
=c_{\beta_k',\alpha_k'+\alpha}r_k'',
\]
where $r_k''=\psi_{\beta_k',\alpha_k'+\alpha}(\psi^{-1}_{\beta_k',\alpha_k'}(r)
c_{\alpha_k',\alpha})$. Therefore,
\[
\lc(\bfe{g})c_{\beta,\alpha}=\sum_{k=1}^{s}\lc(\bfe{g}_k')
c_{\beta_k',\alpha_k'}r_k'c_{\beta,\alpha}=
\sum_{k=1}^{s}\lc(\bfe{g}_k')c_{\beta_k',\alpha_k'+\alpha}r_k''.
\]
Since  $\alpha+\beta=\exp(\lm(\bfe{f}))$, then $\alpha+\alpha_{k}'+\beta_{k}'=\exp(\lm(\bfe{f}))$. Further, if there exists
$\bfe{g}_{w}\in \{\bfe{g}_1,\ldots,\bfe{g}_t\}$ such that $\bfe{g}_w=\bfe{g}_z'$ for some $z\in \{1,\ldots,s\}$, then
$\beta_w=\beta_z'$ and $\alpha+\alpha_z'=\alpha_w$; therefore, in the representation of $\lc(\bfe{f})$ would appear the term $\lc(\bfe{g}_w)c_{\beta_w,\alpha_w}(r_w+r_z''r_j)$. Hence we conclude that $\bfe{f}$ is reducible with respect to $G'$ and consequently $G'$ is a Gröbner basis for $M$.
\end{proof}

\section{Buchberger's algorithm for right modules}\label{5.5.4}

\noindent Recall that we are assuming that $A$ is a bijective skew $PBW$ extension, we will prove in the present section that every submodule $M$ of $A^m$ has a Gröbner basis, and also we will construct the Buchberger's algorithm for computing such bases.

We start this section by fixing some notation and by proving a couple of preliminary results used later.
\begin{definition}\label{BFformodules}
Let $F := \{\bfe{g}_1,\dots,\bfe{g}_s\}\subseteq A^{m}$, $\bfe{X}_F$ the least common multiple of $\{\lm(\bfe{g}_1),\dots ,\lm(\bfe{g}_s)\}$, $\theta\in \mathbb{N}^n$, $\beta_i:=\exp(\lm(\bfe{g}_i))$ and $\gamma_i\in
\mathbb{N}^n$ such that $\gamma_i+\beta_i=\exp(\bfe{X}_F)$, $1\leq i\leq s$. $B_{F,\theta}$ will denote a finite set of generators in $R^s$ of right $R$-module
\begin{center}
$S_{F,\theta}:=\Syz^{r}_R[\lc(\bfe{g}_1)c_{\beta_1,\gamma_1+\theta}
\ \cdots \ \lc(\bfe{g}_s)c_{\beta_s,\gamma_s+\theta}]$.
\end{center}
For $\theta=\textbf{\emph{0}}:=(0,\dots,0)$, $S_{F,\theta}$ will be denoted by $S_F$ and $B_{F,\theta}$ by $B_F$.
\end{definition}
\begin{remark}\label{RemarkSyzygy}
Let $(b_1,\ldots,b_{s})\in S_{F,\theta}$. Since $A$ is bijective, then there exists an unique
$(b_1',\ldots,b_s')\in S_{F}$ such that
\begin{equation}\label{eq5.2}
b_i=\psi_{\beta_i,\gamma_i+\theta}(\psi_{\beta_i,\gamma_i}^{-1}(b_i')c_{\gamma_i,\theta}),
\text{ for each $1\leq i\leq s$.}
\end{equation}
In fact, the existence and uniqueness of $(b_1',\ldots,b_s')$ follows from the bijectivity of $A$. Now, since $(b_1,\ldots,b_{s})\in S_{F,\theta}$, then $\sum_{i=1}^slc(\bfe{g}_i)c_{\beta_i,\gamma_i+\theta}b_i=0$. Replacing $b_i$ of (\ref{eq5.2}) in the last equation, we obtain
\[
\sum_{i=1}^slc(\bfe{g}_i)c_{\beta_i,\gamma_i+\theta}
\psi_{\beta_i,\gamma_i+\theta}(\psi_{\beta_i,\gamma_i}^{-1}(b_i')c_{\gamma_i,\theta})=0.
\]
The equation (\ref{identitypsi}) of Remark \ref{defPsi}, yields
\[
c_{\beta_i,\gamma_i+\theta}
\psi_{\beta_i,\gamma_i+\theta}(\psi_{\beta_i,\gamma_i}^{-1}(b_i')c_{\gamma_i,\theta})=
c_{\beta_i,\gamma_i}b_i'c_{\beta_i+\gamma_i,\theta}.
\]
Thus, $\sum_{i=1}^slc(\bfe{g}_i)c_{\beta_i,\gamma_i}b_i'c_{\beta_i+\gamma_i,\theta}=0$, and since
$c_{\beta_i+\gamma_i,\theta}=c_{\bfe{X}_F,\theta}$ is invertible, then \linebreak $\sum_{i=1}^slc(g_i)c_{\beta_i,\gamma_i}b_i'=0$, i.e., $(b_1',\ldots,b_s')\in S_F$.
\end{remark}
\begin{lemma}\label{SumForModule}
Let $\bfe{g}_1,\ldots,\bfe{g}_s \in A$ , $c_1,\ldots, c_s \in R-\{0\}$ and $\alpha_1,\ldots,\alpha_{s},\beta_1,\ldots,\beta_{s}\in\mathbb{N}^{n}$ such that,
$\bfe{X}_{\delta}:=\lm(\lm(\bfe{g}_i)x^{\alpha_i})$ and $\beta_i:=\exp(\bfe{g}_i),$ for each $1\leq i\leq s$. If $\lm(\sum_{i=1}^{s}\bfe{g}_{i}c_{i}x^{\alpha_{i}})\prec \bfe{X}_{\delta}$, then there exist $r_1,\ldots,r_k\in R$ and $z_1,\ldots,z_s\in A$ such that
\[
\sum_{i=1}^{s}\bfe{g}_{i}c_{i}x^{\alpha_{i}}=\sum_{j=1}^{k}\biggl(\sum_{i=1}^{s}\bfe{g}_{i}
\psi_{\beta_i,\gamma_i}^{-1}(b_{ji})x^{\gamma_{i}}\biggr)r_{j}x^{\delta-\exp(\bfe{X}_{F})}+
\sum_{i=1}^{s}\bfe{g}_{i}z_{i},
\]
where $\bfe{X}_{F}$ is the least common multiple of $\lm(\bfe{g}_{1}),\ldots, \lm(\bfe{g}_{s})$, $\gamma_{i}\in \mathbb{N}^{n}$ is such that $\gamma_{i}+\beta_i=\exp(\bfe{X}_{F})$, for each $1\leq i\leq s$, and
\[
B_{F}=\{\bfe{b}_1,\dots,\bfe{b}_k\}:=\{(b_{11},\dots,b_{1s}),\dots,(b_{k1},\dots,b_{ks})\}.
\]
Moreover, $\lm(\sum_{i=1}^{s} \bfe{g}_{i}\psi_{\beta_i,\gamma_i}^{-1}(b_{ji})x^{\gamma_{i}}r_jx^{\delta-\exp(\bfe{X}_{F})}) \prec \bfe{X}_{\delta}$ for every $1\leq j\leq k$, and $\lm(\bfe{g}_{i}z_{i})\prec\bfe{X}_{\delta}$ for every $1\leq i\leq s$.
\end{lemma}
\begin{proof}
Since $\bfe{X}_{\delta}=\lm(\lm(\bfe{g}_i)x^\alpha_i)$, then $\lm(\bfe{g}_i)\mid\bfe{X}_{\delta}$ and hence $\bfe{X}_F\mid\bfe{X}_{\delta}$, so there exists $\theta\in\mathbb{N}^n$ such that $\exp(\bfe{X}_F)+\theta=\delta$. On the other hand, $\gamma_i+\beta_i=\exp(\bfe{X}_F)$ and
$\alpha_i+\beta_i=\delta$, so $\alpha_i=\gamma_i+\theta$ for each $1\leq i\leq s$. Now,  $\lm(\sum_{i=1}^{s}\bfe{g}_{i}c_{i}x^{\alpha_{i}})\prec\bfe{X}_{\delta}$ implies that $\sum_{i=1}^{s}\lc(\bfe{g}_{i})\sigma^{\beta_i}(c_{i})c_{\beta_{i},\alpha_{i}}=0$. The equation (\ref{id1Psi}) of Remark \ref{defPsi}, yields  $\sum_{i=1}^{s}\lc(\bfe{g}_{i})c_{\beta_i,\alpha_i}d_i=0$, with $d_i=\psi_{\beta_i,\alpha_i}(c_i)$, for each $1\leq i\leq s$. So, $\sum_{i=1}^{s}\lc(\bfe{g}_{i})c_{\beta_i,\gamma_i+\theta}d_i=0$. This implies that $(d_{1},\ldots,d_{s})\in S_{F,\theta}$. By Remark \ref{RemarkSyzygy}, there exists a unique $(d'_1,\ldots,d'_s)\in S_{F}$ such that
$d_{i}=\psi_{\beta_i,\gamma_i+\theta}(\psi_{\beta_i,\gamma_i}^{-1}(d_i')c_{\gamma_i,\theta})$. Then,
\[
c_i=\psi_{\beta_i,\alpha_i}^{-1}(d_i)=\psi_{\beta_i,\gamma_i+\theta}^{-1}(d_i)=
\psi_{\beta_i,\gamma_i}^{-1}(d_i')c_{\gamma_i,\theta},
\]
and therefore, we have
\begin{align*}
\sum_{i=1}^{s}\bfe{g}_{i}c_{i}x^{\alpha_{i}} & = \sum_{i=1}^{s}\bfe{g}_{i}\psi_{\beta_i,\gamma_i}^{-1}(d_i')c_{\gamma_i,\theta}x^{\alpha_{i}}\\
&=\sum_{i=1}^{s}\bfe{g}_{i}\psi_{\beta_i,\gamma_i}^{-1}(d_i')c_{\gamma_i,\theta}x^{\gamma_i+\theta}\\
&=\sum_{i=1}^{s}\bfe{g}_{i}\psi_{\beta_i,\gamma_i}^{-1}(d_i')(x^{\gamma_i}x^{\theta}-p_{\gamma_i,\theta})\\
&=\sum_{i=1}^{s}\bfe{g}_{i}\psi_{\beta_i,\gamma_i}^{-1}(d_i')x^{\gamma_i}x^{\theta}+
\sum_{i=1}^{s}\bfe{g}_{i}p_i
\end{align*}
with $p_{i}=0$ or $\lm(p_i)\prec x^{\theta+\gamma_{i}}.$ Hence, $\bfe{g}_{i}p_{i}=0$ or $\lm(\bfe{g}_ip_i)\prec x^{\theta+\gamma_{i}+\beta_{i}}=\bfe{X}_{\delta}$. On the
other hand, since  $(d'_1,\ldots,d'_s)\in S_{F}$, then there exist $r'_{1},\ldots,r'_{k}\in R$ such
that $(d'_1,\ldots,d'_s)=\bfe{b}_{1}r'_{1}+\cdots+\bfe{b}_{k}r'_{k}=
(b_{11},\ldots,b_{1s})r'_1+\cdots+(b_{k1},\ldots,b_{ks})r'_k$, thus $d'_{i}=\sum_{j=1}^{k}b_{ji}r'_{j}$.

\noindent Therefore
\begin{align*}
\sum_{i=1}^{s}\bfe{g}_{i}\psi_{\beta_i,\gamma_i}^{-1}(d_i')x^{\gamma_i}x^{\theta}
&=\sum_{i=1}^{s}\bfe{g}_{i}\psi_{\beta_i,\gamma_i}^{-1}\Bigl(\sum_{j=1}^{k}b_{ji}r'_{j}
\Bigr)x^{\gamma_i}x^{\theta}\\
&=\sum_{i=1}^{s}\bfe{g}_{i}\Bigl(\sum_{j=1}^{k}\psi_{\beta_i,\gamma_i}^{-1}(b_{ji})
\psi_{\beta_i,\gamma_i}^{-1}(r'_{j})\Bigr)x^{\gamma_i}x^{\theta}\\
&=\sum_{i=1}^{s}\bfe{g}_{i}\Bigl(\sum_{j=1}^{k}\psi_{\beta_i,\gamma_i}^{-1}(b_{ji})(
x^{\gamma_i}\sigma^{-\gamma_i}(\psi_{\beta_i,\gamma_i}^{-1}(r'_{j}))+q_{ij}')\Bigr)x^{\theta},
\end{align*}
with $q_{ij}'=0$ or $\lm(q_{ij}')\prec x^{\gamma_i}.$ Since
$\psi^{-1}_{\beta_i,\gamma_i}(r)=\sigma^{\gamma_i}(\sigma^{-(\gamma_i+\beta_i)}(r))=
\sigma^{\gamma_i}(\sigma^{-\exp(\bfe{X}_F)}(r))$, then we obtain
\begin{align*}
\sum_{i=1}^{s}\bfe{g}_{i}\psi_{\beta_i,\gamma_i}^{-1}(d_i')x^{\gamma_i}x^{\theta}
&=\sum_{i=1}^{s}\bfe{g}_{i}\Bigl(\sum_{j=1}^{k}\psi_{\beta_i,\gamma_i}^{-1}(b_{ji})(x^{\gamma_i}
\sigma^{-\exp(\bfe{X}_F)}(r'_j)+q_{ij}')\Bigr)x^{\theta}\\
&=\sum_{j=1}^{k}\sum_{i=1}^{s}\bfe{g}_i\psi_{\beta_i,\gamma_i}^{-1}(b_{ji})x^{\gamma_i}r_jx^{\theta}+
\sum_{i=1}^{s}\sum_{j=1}^{k}\bfe{g}_iq_{ij}x^{\theta}\\
&=\sum_{j=1}^{k}\Bigl(\sum_{i=1}^{s}\bfe{g}_i\psi_{\beta_i,\gamma_i}^{-1}(b_{ji})x^{\gamma_i}\Bigr)
r_jx^{\theta}+\sum_{i=1}^{s}\bfe{g}_{i}q_{i},
\end{align*}
where $q_{ij}:=\psi_{\beta_i,\gamma_i}^{-1}(b_{ji})q_{ij}',$ $r_j:=\sigma^{-\exp(\bfe{X}_F)}(r_j')$ and so, $q_{i}:=\sum_{j=1}^{k}q_{ij}x^{\theta}=0$ or $\lm(q_{i})\prec x^{\theta+\gamma_{i}}$. Finally we get,
\[
\sum_{i=1}^{s}\bfe{g}_{i}c_{i}x^{\alpha_{i}}=\sum_{j=1}^{k}\biggl(\sum_{i=1}^{s}\bfe{g}_{i}
\psi_{\beta_i,\gamma_i}^{-1}(b_{ji})x^{\gamma_{i}}\biggr)r_{j}x^{\delta-\exp(\bfe{X}_{F})}+
\sum_{i=1}^{s}\bfe{g}_{i}z_{i},
\]
with $z_{i}:=p_{i}+q_{i}$ for $1\leq i\leq s$. Is easy to see that
$\lm(\sum_{i=1}^{s}\bfe{g}_i\psi_{\beta_i,\gamma_i}^{-1}(b_{ji})x^{\gamma_{i}}r_{j}x^{\theta})\prec \bfe{X}_{\delta}$ since $\lm(\sum_{i=1}^{s}\bfe{g}_i\psi_{\beta_i,\gamma_i}^{-1}(b_{ji})x^{\gamma_{i}}) \prec x^{\gamma_{i}+\beta_{i}}$, and $\lm(\bfe{g}_iz_i)=\lm(\bfe{g}_ip_i+\bfe{g}_iq_i)\prec\bfe{X}_{\delta}$.
\end{proof}

Under the notation used in Definition \ref{BFformodules} and Lemma \ref{SumForModule}, we will prove the main result of the present section.

\begin{theorem}\label{buchbergerformod}
Let $M\neq 0$ be a right submodule of $A^{m}$ and let $G$ be a finite subset of non-zero generators of $M$. Then the following conditions are equivalent.
\begin{enumerate}
\item[\rm (i)] $G$ is a Gröbner basis of $M$.
\item[\rm (ii)] For all $F:=\{\bfe{g}_1,\dots,\bfe{g}_s\}\subseteq G$, with $\bfe{X}_F\neq\bfe{0}$ and for any $(b_1,\dots,b_s)\in B_{F}$, we have
\[
\sum_{i=1}^s\bfe{g}_i\psi_{\beta_i,\gamma_i}^{-1}(b_i)x^{\gamma_i}\xmapsto{\,\, G\,\, }_+ \bfe{0}.
\]
\end{enumerate}
\begin{proof}
{\rm (i)} $\Rightarrow$ {\rm (ii)}: We observe that $\bfe{f}:=\sum_{i=1}^s\bfe{g}_i\psi_{\beta_i,\gamma_i}^{-1}(b_i)x^{\gamma_i}\in M.$ Then Theorem
\ref{teogrobnersigmapbwformodules} yields $\bfe{f}\xmapsto{\,\, G\,\, }_+\bfe{0}$.

\smallskip\noindent
{\rm (ii)} $\Rightarrow$ {\rm (i)}: Assume that $\bfe{0}\neq\bfe{f}\in M.$ We will prove that the condition (iii) of Theorem \ref{teogrobnersigmapbwformodules} holds. If $G:=\{\bfe{g}_1,\dots,\bfe{g}_t\}$, then there exist $h_1,\dots,h_t\in A$ such that $\bfe{f}=\bfe{g}_1h_1+\cdots+\bfe{g}_th_t$ and we can choose $\{h_i\}_{i=1}^t$ such that $\bfe{X}_{\delta}:=\max\{\lm(\lm(\bfe{g}_i)\lm(h_i))\}_{i=1}^t$ is minimal. Let $\lm(h_i):=x^{\alpha_i}$, $c_{i}:=\lc(h_{i})$, $\lm(\bfe{g}_i):=x^{\beta_i}$ for $1\leq i\leq t$ and $F:=\{\bfe{g}_i\in G\mid \lm(\lm(\bfe{g}_i)\lm(h_i))=\bfe{X}_{\delta}\}.$ Up to a renumbering the elements of $G,$ we can assume that $F=\{\bfe{g}_1,\dots,\bfe{g}_s\}$. We will consider two possible cases.

\medskip
\textit{Case 1}: $\lm(\bfe{f})=\bfe{X}_{\delta}$. Then $\lm(\bfe{g}_i)\mid \lm(\bfe{f})$ for
$1\leq i\leq s$ and
\[
\lc(\bfe{f})=\sum_{i=1}^{s}\lc(\bfe{g}_i)\sigma^{\beta_i}(c_i)c_{\beta_i,\alpha_i}
=\sum_{i=1}^{s}\lc(\bfe{g}_i)c_{\beta_i,\alpha_i}\psi_{\beta_i,\alpha_i}(c_i).
\]
i.e., the condition (iii) of Theorem \ref{teogrobnersigmapbwformodules} holds.

\medskip
\textit{Case 2}: $\lm(\bfe{f})\prec\bfe{X}_{\delta}$. We will prove that this yields a contradiction. To begin, note that $f$ can be written as
\begin{align}\label{writingf}
\bfe{f}=\sum_{i=1}^{s}\bfe{g}_ic_{i}x^{\alpha_{i}}+\sum_{i=1}^{s}\bfe{g}_i(h_{i}-c_{i}x^{\alpha_{i}})+
\sum_{i=s+1}^{t}\bfe{g}_ih_{i}.
\end{align}
We have $\lm(\bfe{g}_i(h_{i}-c_{i}x^{\alpha_{i}}))\prec\bfe{X}_{\delta}$ for each $1\leq i\leq s$ and $\lm(\bfe{g}_ih_{i})\prec\bfe{X}_{\delta}$ for each $s+1\leq i\leq t.$ Hence
\[
\lm(\sum_{i=1}^{s}\bfe{g}_i(h_{i}-c_{i}x^{\alpha_{i}}))\prec\bfe{X}_{\delta}\text{ and }
\lm(\sum_{i=s+1}^{t}\bfe{g}_ih_{i})\prec\bfe{X}_{\delta},
\]
and $\lm(\sum_{i=1}^{s}\bfe{g}_{i}c_{i}x^{\alpha_{i}})\prec\bfe{X}_{\delta}$. Under the notation used in Lemma
\ref{SumForModule} (and its notation), we have
\begin{align}
\sum_{i=1}^{s}\bfe{g}_{i}c_{i}x^{\alpha_{i}}=\sum_{j=1}^{k}\Bigl(\sum_{i=1}^{s}
\bfe{g}_{i}\psi_{\beta_i,\gamma_i}^{-1}(b_{ji})x^{\gamma_{i}}\Bigr)r_{j}
x^{\delta-\exp(\bfe{X}_{F})}+\sum_{i=1}^{s}\bfe{g}_{i}z_{i},
\end{align}
where $\lm(\sum_{i=1}^{s}\bfe{g}_{i}\psi_{\beta_i,\gamma_i}^{-1}(b_{ji})x^{\gamma_{i}}
x^{\delta-\exp(\bfe{X}_{F})})\prec\bfe{X}_{\delta}$ for each $1\leq j\leq k$ and $\lm(\bfe{g}_{i}z_i)\prec\bfe{X}_{\delta}$ for $1\leq i\leq s$. By the hypothesis,
$\sum_{i=1}^s\bfe{g}_i\psi_{\beta_i,\gamma_i}^{-1}(b_{ji})x^{\gamma_i}
\xmapsto{\,\,G\,\,}_{+}\bfe{0}$, whence, by Theorem \ref{algdivformodules}, there exist $q_1,\dots,q_t\in A$ such that
\[
\sum_{i=1}^s\bfe{g}_i\psi_{\beta_i,\gamma_i}^{-1}(b_{ji})x^{\gamma_i}=\sum_{i=1}^t\bfe{g}_iq_i,
\]
with $\lm(\sum_{i=1}^s\bfe{g}_i\psi_{\beta_i,\gamma_i}^{-1}(b_{ji})x^{\gamma_i})=
\max\{\lm(\lm(\bfe{g}_i)\lm(q_i))\}_{i=1}^t.$ Since $(b_{j1},\dots,b_{js})$ $\in B_{F}$, then using the equation (\ref{id2Psi}) of Remark \ref{defPsi}, we get
\[
\lc(\sum_{i=1}^s\bfe{g}_i\psi_{\beta_i,\gamma_i}^{-1}(b_{ji})x^{\gamma_i})=
\sum_{i=1}^{s}\lc(\bfe{g}_i)\sigma^{\beta_i}(\psi_{\beta_i,\gamma_i}^{-1}(b_{ji}))c_{\beta_i,\gamma_i}
=\sum_{i=1}^{s}\lc(\bfe{g}_i)c_{\beta_i,\gamma_i}b_{ji}=0.
\]
Hence $\lm\Bigl(\sum_{i=1}^s\bfe{g}_i\psi_{\beta_i,\gamma_i}^{-1}(b_{ji})x^{\gamma_i}\Bigr)\prec\bfe{X}_{F}$ and $\lm(\lm(\bfe{g}_i)\lm(q_i))\prec\bfe{X}_{F}$ for each $1\leq i\leq t$. Therefore,
\begin{align*}
\sum_{j=1}^{k}\Bigl(\sum_{i=1}^{s}\bfe{g}_i\psi_{\beta_i,\gamma_i}^{-1}(b_{ji})x^{\gamma_{i}}\Bigr)
r_{j}x^{\delta-\exp(\bfe{X}_{F})} &
= \sum_{j=1}^{k}\Bigl(\sum_{i=1}^{t}\bfe{g}_iq_i\Bigr)r_{j}x^{\delta-\exp(\bfe{X}_{F})}\\
& = \sum_{i=1}^{t}\sum_{j=1}^{k}\bfe{g}_iq_ir_{j}x^{\delta-\exp(\bfe{X}_{F})}\\
& = \sum_{i=1}^{t}\bfe{g}_i\widetilde{q}_{i},
\end{align*}
with $\widetilde{q}_{i}:=\sum_{j=1}^{k}q_ir_{j}x^{\delta-\exp(\bfe{X}_{F})}$ and
$\lm(\bfe{g}_i\widetilde{q}_{i})\prec\bfe{X}_{\delta}$ for each $1\leq i\leq t$. Substituting
$\sum_{i=1}^{s}\bfe{g}_{i}c_{i}x^{\alpha_{i}}=\sum_{i=1}^{t}\bfe{g}_i\widetilde{q}_{i}+
\sum_{i=1}^{s}\bfe{g}_{i}z_i$ into equation (\ref{writingf}), we obtain\vspace*{-1mm}
\[
f=\sum_{i=1}^{t}\bfe{g}_i\widetilde{q}_{i}+\sum_{i=1}^{s}\bfe{g}_i(h_{i}-c_{i}x^{\alpha_{i}})+
\sum_{i=1}^{s}\bfe{g}_iz_{i}+\sum_{i=s+1}^{t}\bfe{g}_ih_{i},
\]
and so we have expressed $\bfe{f}$ as a combination of vectors $\bfe{g}_{1},\ldots,\bfe{g}_{t}$, where each of its terms has leading monomial $\prec\bfe{X}_{\delta}$. This contradicts the minimality of $\bfe{X}_{\delta}$ and finishes the proof.
\end{proof}
\end{theorem}
\begin{corollary}\label{algorithmforbijectivemodules}
Let $F=\{\bfe{f}_1,\dots ,\bfe{f}_s\}$ be a set of non-zero vectors of $A^{m}$. The algorithm below (Algorithm \ref{algoBuchbergerinAm}) produces a Gr\"obner basis for the right submodule $\langle F\rangle$ of $A^{m},$ where $P(X)$ denotes the set of subsets of the set $X$.

\begin{algorithm}[h!]\small
\caption{Right Buchberger's algorithm in $A^{m}$}
\KwIn{$F := \{\bfe{f}_1,\dots,\bfe{f}_s\}\subseteq A^{m}$, $\bfe{f}_i\neq \bfe{0}$, $1\leq i\leq s$}
\KwOut{$G=\{\bfe{g}_1,\dots ,\bfe{g}_t\}$ a Gröbner basis for $\langle F\rangle$}
\Ini $G\gets\emptyset, G'\gets F$\;
\While{$G'\neq G$}{
 $D\gets P(G')-P(G)$;\\
 $G\gets G'$;\\
 \For{each $S:=\{\bfe{g}_{i_1},\dots ,\bfe{g}_{i_k}\}\in D$ with $\bfe{X}_S\neq\bfe{0}$}{
  Compute $B_S$;\\
  \For{ each $\bfe{b}=(b_1,\dots ,b_k)\in B_S$}{
   Reduce $\sum_{j=1}^k\bfe{g}_{i_j}\psi_{\beta_{i_j},\gamma_{j}}^{-1}(b_j)
   x^{\gamma_j}\xmapsto{\,\, G'\,\, }_+\bfe{r}$; \# {\scriptsize with $r$ reduced with respect to $G'$; $\beta_{i_j},\gamma_j$ as in Def \ref{BFformodules}}\\
   \If{$\bfe{r}\ne\bfe{0}$}{
     $G'\gets G'\cup \{\bfe{r}\}$;
   }
  }
 }
}
\label{algoBuchbergerinAm}
\end{algorithm}
\end{corollary}

\noindent We finish this section by the following useful result.
\begin{corollary}\label{existence}
Every right submodule of the free $A$-module $A^{m}$ has a Gröbner basis.
\end{corollary}
\begin{proof}
Apply Theorem \ref{buchbergerformod} and Corollary \ref{algorithmforbijectivemodules}.
\end{proof}

\section{Examples implemented in \texttt{SPBWE} library}

The extensions skew $PBW$ extensions were implemented in \textrm{Maple} with the development of the \texttt{SPBWE} library (see \cite{Fajardo1}), which allows to make important computations with this type of rings and can also provide answers to several homological problems such as the computation of syzygies; within the library are already developed the algorithms that we present in this paper:  the right division algorithm and the right Buchberger algorithm, below here we will present only a brief view of its execution.

\begin{example}\label{DivisionAlgDifussionAlg}
Consider the diffusion algebra $A:=\sigma(\mathbb{Q}[x_1,x_2])\langle D_1,D_2\rangle$ subject to relation:
\[
D_2D_1=2D_1D_2+x_2D_1-x_1D_2.
\]
Taking the following polynomials in $A$
\[
f:=x_1x_2^2D_1^2D_2+x_1^2x_2D_2:\quad
f_1:=x_1^2x_2D_1D_2:\quad
f_2:=x_2D_1:\quad
f_3:=x_2D_1:
\]
We use the right division algorithm over these polynomials as follow
and we get polynomials $g_1=\tfrac{1}{2}x_2D_1+\tfrac{1}{2}x_1x_2,$ $g_2=-\tfrac{x_1x_2^2D_1}{2}$ and $g_3=x_1x_2,$ such that
$$ f=f_1g_1+f_2g_2+f_3g_3. $$
Therefore,
\[
x_1x_2^2D_1^2D_2+x_1^2x_2D_2\in\langle x_1^2x_2D_1D_2,x_2D_1,x_2D_1\rangle_A
\]
with
\[
x_1x_2^2D_1^2D_2+x_1^2x_2D_2=x_1^2x_2D_1D_2
(\tfrac{1}{2}x_2D_1+\tfrac{1}{2}x_1x_2)+
x_2D_1\Bigl(-\tfrac{x_1x_2^2D_1}{2}\Bigr)+x_2D_1(x_1x_2).
\]
\end{example}
The following example is a non-trivial instance of applicability of the \texttt{SPBWE} library. In particular, it is possible to define iterated skew $PBW$ extensions in the library and compute left or right Gröbner bases over theses.
\begin{example}\label{IteratedSPBWE01}
Let $A=\mathbb{C}[w,\varphi]$ the skew polynomial ring of endomorphism type with $\varphi(q)=\overline{q},$ for $q\in\mathbb{C}.$ Using the \texttt{SPBWE}, we can to define the extension $C=\sigma(A)\langle x,y,z\rangle,$ subjects to relations
\[  yx=2xy,\quad zx=4xz-x,\quad zy=4yz-y,\]
with $A$-endomorphisms $\sigma_i:$ $\sigma_1(w)=2w,$ $\sigma_2(w)=3w,$ $\sigma_3(w)=w,$ and $\sigma_i$-derivations $\delta_i=0$ for $i=1,2,3$.

Consider the right submodule $M:=\langle$\textbf{\emph{$f_1$}}, \textbf{\emph{$f_2$}}, \textbf{\emph{$f_3$}}, \textbf{\emph{$f_4$}}$\rangle_C\subseteq C^4,$ with
\begin{align*}
\emph{\textbf{f}}_1 &= (-y^2,-w y+y,y,w x-x y),\\
\emph{\textbf{f}}_2 &= (-w y-y,x y+y^2-w,w^2,w^2+w x+w y),\\
\emph{\textbf{f}}_3 &= (-x+1,-w y+x^2+x y,w^2+w x+w,x),\\
\emph{\textbf{f}}_4 &= (y^2+x,w x+1,0,w^2+w y-y^2).
\end{align*}
Let $\textbf{\emph{$v$}}:=\emph{\textbf{f}}_1p_1+\emph{\textbf{f}}_2p_2+
\emph{\textbf{f}}_3p_3+\emph{\textbf{f}}_4p_4\in M,$ with
$p_1:=76x+95z,$ $p_2:=xz,$ $p_3:=w-iy$ and $p_4:=iz,$ next, we will use the \texttt{SPBWE} library, in particular, the right division algorithm and the Buchberger algorithm over $C,$ to verify that \textbf{\emph{$v$}} lies in $M.$

First, we use the right division algorithm on $V$ and $M$ as follow
\begin{align*}
& V:=\begin{bsmallmatrix}
-304 x y^2+(-2 w-2) x y z+(-95+I) y^2 z+I x y+I x z-2 w x-I y+w\\
2x^2yz+4xy^2z-Ix^2y-Ixy^2+4wx^2+(-146w+152)xy+(-1-I)wxz-Iwy^2+
(-95w+95)yz-3w^2y+Iz\\
(Iw+152)xy+w^2xz+95yz+2w^2x+(-Iw^2+Iw)y+w^3+w^2\\
-152x^2y+wx^2z+(2w-95)xyz-Iy^2z+76wx^2-Ixy+(w^2+95w)x z-Iwyz+2wx+Iw^2z
\end{bsmallmatrix}:\\
& M:=\begin{bmatrix}
-y^2 & -w y+y & y & wx-xy\\
-wy-y & xy+y^2-w & w^2 & w^2+wx+wy\\
-x+1 & -wy+x^2+xy & w^2+wx+w & x\\
y^2+x & wx+1 & 0 & w^2+wy-y^2
\end{bmatrix}:\\
& >\quad \texttt{DivisionAlgorithm($V,$ $M,$ gradlexrev, TOP, $C,$ right)}
\end{align*}
We obtain four polynomials $q_1=76x+95z,$ $q_2=xz-\tfrac{1}{2}Ix+Iy,$ $q_3=q_4=0$ and a vector $$\emph{\textbf{h}}=\begin{bmatrix}
Iy^2z+Iwxy+Ixz+(-Iw+I)y^2-2wx-Iy+w\\
-Iy^3+4wx^2+6wxy-Iwxz-Iwy^2+\tfrac{1}{2}Iwx+(-3w^2-Iw)y+Iz\\
Iwxy+\bigl(2+\tfrac{1}{2}I\bigr)w^2x+(-2Iw^2+Iw)y+w^3+w^2\\
-Iy^2z-\tfrac{1}{2}Iwx^2-Ixy+Iwy^2-Iwyz+\bigl(\tfrac{1}{2}Iw^2+2w\bigr)x-
Iw^2y+Iw^2z
\end{bmatrix}\in C^{4}$$ such that
$$\emph{\textbf{f}}=\emph{\textbf{f}}_1q_1+\emph{\textbf{f}}_2q_2+
\emph{\textbf{f}}_3q_3+\emph{\textbf{h}}.$$

Since $\emph{\textbf{h}}\neq\emph{\textbf{0}}$, we have a
second option. For this purpose, we use the following statement in \textrm{Maple}
$$ >\quad G:=\texttt{BuchbergerAlgSkewPoly($M$, gradlex, TOP, $C$, right)}$$
We obtain a Gröbner basis of $M,$ $G=\{\emph{\textbf{h}}_1,\emph{\textbf{h}}_2,
\emph{\textbf{h}}_3,\emph{\textbf{h}}_4,\emph{\textbf{h}}_5,
\emph{\textbf{h}}_6,\emph{\textbf{h}}_7,\emph{\textbf{h}}_8\},$ with
\[
\emph{\textbf{h}}_1=\emph{\textbf{f}}_1,\  \emph{\textbf{h}}_2=\emph{\textbf{f}}_2,\
\emph{\textbf{h}}_3=\emph{\textbf{f}}_3,\
\emph{\textbf{h}}_4=\emph{\textbf{f}}_4,
\]
\begin{align*}
&\emph{\textbf{h}}_5=
\begin{bsmallmatrix}
-xy^2-y^3-\tfrac{1}{4}x^2\\ -\tfrac{1}{4}wx^2+(-w+1)y^2-\tfrac{1}{4}x\\ y^2\\ \tfrac{1}{2}wxy-\tfrac{1}{4}w^2x
\end{bsmallmatrix},\
\emph{\textbf{h}}_6=
\begin{bsmallmatrix}
-2wxy+(2w+2)y^2-2 y\\ -2y^3+2wy^2-wx+2wy\\ -2wxy+w^2x+(-4w^2-2w)y\\
wx^2-2xy-2wy^2+w^2x-2w^2y
\end{bsmallmatrix},\\
&\emph{\textbf{h}}_7=
\begin{bsmallmatrix}
-\tfrac{1}{4}wx^2y+(-\tfrac{1}{2}w+\tfrac{1}{2})xy^2-\tfrac{1}{4}xy-y^2\\
\tfrac{1}{2}wxy^2+wy^3-\tfrac{1}{16}wx^2-\tfrac{1}{4}wxy\\
-\tfrac{1}{4}wx^2y-wxy^2+\tfrac{1}{16}w^2x^2-\tfrac{1}{4}wxy+(-w^2-w)y^2\\
\tfrac{1}{16}wx^3+(\tfrac{1}{2}w-\tfrac{1}{4})x^2y+(\tfrac{1}{2}w-1)xy^2+
\tfrac{1}{16}w^2x^2+\tfrac{1}{4}w^2xy
\end{bsmallmatrix},\\
&\emph{\textbf{h}}_8=
\begin{bsmallmatrix}
\tfrac{5}{4}wx^2y^2+(4w-2)xy^3+\tfrac{1}{2}xy^2+2y^3\\
-wxy^3-2wy^4+(\tfrac{1}{16}w^2+\tfrac{1}{16}w)x^2y+(\tfrac{1}{2}w^2-
\tfrac{1}{2}w+\tfrac{1}{2})xy^2\\
\tfrac{1}{2}wx^2y^2+2wxy^3+(-\tfrac{1}{8}w^2-\tfrac{1}{16}w)x^2y+
\tfrac{1}{2}xy^2+(2w^2+2w)y^3\\
(-w+2)xy^3-\tfrac{1}{96}w^2x^3+(-\tfrac{7}{24}w^2+\tfrac{1}{4}w)x^2y-
\tfrac{1}{2}w^2xy^2
\end{bsmallmatrix}.
\end{align*}
Finally, using the statement
$$ >\quad \texttt{DivisionAlgorithm($V,$ $G,$ gradlex, TOP, $C,$ right)}$$ we obtain eight polynomials $q_1=76x+95z,$ $q_2=xz-1/2Ix+Iy,$ $q_3=w,$ $q_4=iz,$ $q_6=1/2i$ and $q_5=q_7=q_8=0$ such that
\[
\emph{\textbf{h}}=\emph{\textbf{h}}_1q_1+\emph{\textbf{h}}_2q_2+
\emph{\textbf{h}}_3q_3+\emph{\textbf{h}}_4q_4+\emph{\textbf{h}}_5q_5+
\emph{\textbf{h}}_6q_6+\emph{\textbf{h}}_7q_7+\emph{\textbf{h}}_8q_8.
\]
Consequently, the vector $h$ lies in $M.$
\end{example}

\section{Future perspective}

As consequence of the algorithms presented in this paper over a bijective skew $PBW$ extension $A,$ we can respond to problems of homological algebra such as: computation of the right module of syzygies of a right $A$-module $M;$ computation of a right inverse of rectangular matrix on $A;$ computation of the intersection and quotient for ideals or modules over $A$; computation of the Ext$_A^{r}(M,N),$ where $M$ is a finitely generated left $A$-submodule of $A^{m}$ and $N$ is a finitely generated centralizing $A$-subbimodule of $A^{l}$; among another applications. Now is possible to complete the SPBWE library and provide support in areas of non-commutative algebra that have not yet implemented computationally.


\end{document}